\documentclass{amsart}
\usepackage{latexsym,mathtools}
\usepackage{amscd,amsfonts,amsmath,amssymb,amsthm}
\usepackage[normalem]{ulem}
\usepackage{comment}
\usepackage{tikz,tikz-cd}

\usepackage[colorlinks]{hyperref}
\usepackage{color}
\definecolor{darkgreen}{RGB}{55,138,0}

\hypersetup{citecolor = darkgreen}

\numberwithin{equation}{section}

\theoremstyle{plain}

\newtheorem{theorem}{Theorem}[section]

\newtheorem{thmintro}{Theorem}

\newtheorem*{theorem*}{Theorem}

\newtheorem{lemma}[theorem]{Lemma}
\newtheorem{proposition}[theorem]{Proposition}
\newtheorem{corollary}[theorem]{Corollary}

\theoremstyle{definition}
\newtheorem{definition}[theorem]{Definition}
\newtheorem{example}[theorem]{Example}

\newtheorem{remark}[theorem]{Remark}
\newtheorem{question}[theorem]{Question}

\numberwithin{equation}{theorem} 

\mathchardef\mhyphen="2D

\DeclareMathOperator\Aut{Aut}
\DeclareMathOperator\Div{Div}
\DeclareMathOperator\End{End} 
\DeclareMathOperator\Ext{Ext} 
\DeclareMathOperator\Gal{Gal}
\DeclareMathOperator\GKdim{GKdim}
\DeclareMathOperator\gldim{gldim}
\DeclareMathOperator\gr{gr} 
\DeclareMathOperator\hdet{hdet}
\DeclareMathOperator\Hom{Hom}
\DeclareMathOperator\Id{Id}
\DeclareMathOperator{\Oz}{Oz}
\DeclareMathOperator\p{{\sf p}} 
\DeclareMathOperator\PIdeg{PIdeg}
\DeclareMathOperator\rk{rk}

\newcommand\CC{\mathbb C}

\newcommand\PP{\mathbb P}
\newcommand\RR{\mathbb R}
\newcommand\ZZ{\mathbb Z}

\newcommand\cD{\mathcal D}
\newcommand\cL{\mathcal L}
\newcommand\cO{\mathcal O}

\newcommand\fb{\mathfrak b}
\newcommand\fd{\mathfrak d}
\newcommand\fl{\mathfrak l}
\newcommand\fm{\mathfrak{m}}

\newcommand\bp{\mathbf p}
\newcommand\bq{\mathbf q}
\newcommand\bt{\mathbf t}
\newcommand\bx{\mathbf x}

\newcommand\Autgr{\Aut_{\mathrm{gr}}}
\newcommand\ch{\operatorname{char}}
\newcommand\cyc{\mathbb{Z}}
\newcommand\kOz{\kk\!\Oz}
\newcommand\SP{S_{\bp}}
\newcommand\ZSP{Z(S_{\bp})}

\newcommand\inv{^{-1}}
\newcommand\iso{\cong}
\newcommand\kk{\Bbbk}

\newcommand\grp[1]{{\langle #1 \rangle}}
\newcommand\restrict[1]{\raisebox{-.3ex}{$|$}_{#1}}

\definecolor{orange}{rgb}{0.93, 0.53, 0.18}

\begin{document}

\title[Ozone groups of PI AS regular algebras]
{Ozone groups of Artin--Schelter regular algebras
satisfying a polynomial identity}

\author{Kenneth Chan, Jason Gaddis, Robert Won, James J. Zhang}

\address{(Chan) Department of Mathematics, 
University of Washington, Box 354350, Seattle, Washington 98195, USA}
\email{kenhchan@math.washington.edu, ken.h.chan@gmail.com}

\address{(Gaddis) Department of Mathematics, 
Miami University, Oxford, Ohio 45056, USA} 
\email{gaddisj@miamioh.edu}

\address{(Won) Department of Mathematics,
The George Washington University, Washington, DC 20052, USA}
\email{robertwon@gwu.edu}

\address{(Zhang) Department of Mathematics,
University of Washington, Box 354350, Seattle, Washington 98195, USA}
\email{zhang@math.washington.edu}

\begin{abstract}
We study the ozone group of noetherian Artin--Schelter regular 
algebras satisfying a polynomial identity. 
The ozone group was shown in previous work by the authors to 
be an important invariant in the study of PI skew polynomial 
rings and their centers. In this paper, we show that skew 
polynomial rings are in fact characterized as those algebras 
with maximal rank abelian ozone groups. We also classify those with 
trivial ozone groups, which must necessarily be Calabi--Yau. 
This class includes most three-dimensional PI Sklyanin algebras. 
Further examples and applications are given, including 
applications to the Zariski Cancellation Problem.
\end{abstract}




\maketitle

\setcounter{section}{-1}
\section{Introduction}\label{yysec0}

Every noetherian PI Artin--Schelter regular algebra is a domain 
and a maximal order in its quotient division ring satisfying the 
Auslander--Gorenstein and Cohen--Macaulay conditions \cite{StaZ}. 
The centers of PI Artin--Schelter regular algebras of global 
dimension three have been studied by Artin \cite{Art} and Mori 
\cite{Mor} from the point of view of noncommutative projective 
geometry. The aim of this paper is to study the centers of PI 
Artin--Schelter regular algebras from an algebraic point of 
view. 

Throughout, let $\kk$ denote a base field with $\ch \kk = 0$. 
We assume that all algebras are $\kk$-algebras and 
$\otimes=\otimes_\kk$. If $A$ is an algebra, let $Z(A)$ denote 
its center. If the algebra $A$ is understood we will simply 
let $Z=Z(A)$.

For an algebra $A$ with center $Z$, we define the 
\emph{ozone group} of $A$ over $Z$ to be 
\[ \Oz(A):=\Aut_{Z\text{-alg}}(A).\]
See Definition~\ref{yydef1.1} for a more general definition and a version for graded algebras.

By Theorem~\ref{yythm0.7}, if $A$ is a noetherian PI AS regular algebra, then
\begin{equation}\label{E0.0.1}
1 \leq \left|\Oz(A)\right| \leq \rk_Z(A).
\end{equation}
Both inequalities above can be equalities 
(see Example~\ref{yyexa2.1}, Theorem~\ref{yythm2.6}, 
Example~\ref{yyexa1.2}, and Theorem~\ref{yythm0.8}).

\begin{definition}\label{yydef0.1}
A connected graded algebra $A$ is called \emph{Artin--Schelter 
Gorenstein} (or \emph{AS Gorenstein}, for short) if $A$ has 
injective dimension $d<\infty$ on the left and on the right, 
and
\[\Ext^i_A({}_A\kk, {}_{A}A)\cong \Ext^i_{A}(\kk_A,A_A)\cong 
\delta_{id} \kk(\fl)\]
where $\delta_{id}$ is the Kronecker-delta. If in addition $A$ 
has finite global dimension and finite Gelfand--Kirillov 
dimension, then $A$ is called \emph{Artin--Schelter regular} 
(or \emph{AS regular}, for short) of dimension $d$.
\end{definition}

For an algebra $A$, the \emph{enveloping algebra} of $A$ is 
$A^e=A \otimes A^{\mathrm{op}}$. If $\sigma$ is an 
automorphism of $A$, then $\prescript{\sigma}{}{A}$ is the 
$A^e$-module which, as a $\kk$-vector space is just $A$, but 
where the natural action is twisted by $\sigma$ on the left: 
that is,
\[ (a \otimes b)\cdot c = \sigma(a)cb\]
for all $a \otimes b \in A^e$ and $c \in A$.

\begin{definition}
\label{yydef0.2}
Suppose $A$ is AS regular of dimension $d$. Then there is a 
graded algebra automorphism $\mu_A$ of $A$, called the 
\emph{Nakayama automorphism}, such that
\[ \Ext_{A^e}^d(A,A^e) \cong\prescript{\mu_A}{}{A}.\]
When $\mu_A = \Id$, then we say that $A$ is \emph{Calabi--Yau} 
(CY).
\end{definition}

\begin{thmintro}[Theorem~\ref{yythm0.3}]
\label{yythmintro0.3}
Consider the following conditions for a noetherian
PI AS regular algebra $A$. 
\begin{enumerate}
\item[(1)] $\Oz(A)=\{1\}$. 
\item[(2)] Every normal element of $A$ is central.
\item[(3)] $A$ is CY.
\item[(4)] $Z$ is Gorenstein.
\end{enumerate}
Then $(1) \Leftrightarrow (2) \Rightarrow (3) \Rightarrow (4)$.
\end{thmintro}

This theorem suggests that $A$, $Z$, and $\Oz(A)$ are deeply 
related. As a consequence of $(2) \Rightarrow (4)$, if the 
center of $A$ is not Gorenstein, then $A$ has a non-central
normal element.

In Section~\ref{yysec1} we give several examples of PI AS 
regular algebras and their ozone groups. We also present 
some general results about the behavior of ozone groups 
under certain extensions. This leads to the following theorem.

\begin{thmintro}[Theorem 
\ref{yythm0.4}]
\label{yythmintro0.4}
Let $G$ be a finite abelian group. Then there exists a 
noetherian PI AS regular algebra $A$ such that $\Oz(A)=G$.
\end{thmintro}

In Section~\ref{yysec2}, we turn our attention to Sklyanin algebras.
Let $[a:b:c] \in \PP^2_\kk\backslash\mathcal{D}$ where
\[ \cD = \{[1:0:0], [0:1:0], [0:0:1] \} \cup \{ [a : b : c] : a^3 = b^3 = c^3=1\}.\]
The \emph{three-dimensional Sklyanin algebra} $S(a,b,c)$ is the $\kk$-algebra
\begin{align}\label{E0.4.1}
\kk\langle x,y,z \rangle/(axy + byx + cz^2, yz + bzy + cx^2, azx + bxz + cy^2).
\end{align}
Formally, we have defined the \emph{nondegenerate} Sklyanin algebras while a \emph{degenerate} Sklyanin algebra is one in which $[a:b:c] \in \cD$. In this paper we consider only the nondegenerate case.

It is well-known that $S=S(a,b,c)$ is an AS regular domain of global 
dimension three. Moreover, there exists a canonical central element $g \in S$ such that $S/gS$ is isomorphic to a twisted homogeneous coordinate ring $B(E,\cL, \sigma)$ where $E \subseteq \PP(A_1)=\PP^2$ is a cubic in $\PP^2$ or all of $\PP^2$, $\cL = \cO_{\PP(A_1)}\restrict{E}$, and $\sigma \in \Aut(E)$. We say that $S$ is \emph{elliptic} if $E$ is a smooth elliptic curve.
The Sklyanin algebra $S$ is PI if and only if $\sigma$ has finite order $n$. In this case, if $3 \nmid n$, then $S$ has rank $n^2$ over its center, while if $3 \mid n$, then $S$ has rank $(n/3)^2$ over its center \cite{AST}.

We compute the ozone groups of PI Sklyanin algebras, showing that, in most cases, a 3-dimensional 
PI Sklyanin algebra has trivial ozone group.

\begin{thmintro}[Theorem \ref{yythm0.5}]
\label{yythmintro0.5}
Let $\kk$ be algebraically closed.
Let $A$ be a PI three-dimensional Sklyanin algebra whose defining automorphism has order $n \geq 2$. 
Suppose that $A$ is not a skew polynomial ring.
\begin{enumerate}
\item[(1)] 
If $n\geq2$ and $3 \nmid n$, then $\mathrm{Oz}(A)=\{1\}$.
\item[(2)] 
If $n>3$ and $3 \mid n$, then:
\begin{enumerate}
    \item if $A$ is elliptic, then $\Oz(A) \iso \cyc_3$, and
    \item if $A$ is not elliptic, then $\Oz(A) \neq \{1\}$.
\end{enumerate}
\item[(3)]
If $n=3$, then 
$A=S(1,0,\omega)$ where $\omega$ is a primitive sixth root of unity and $\Oz(A)=\ZZ_3$.
\end{enumerate}
\end{thmintro}

The property of having a trivial ozone group is not unique to 
the Sklyanin algebras in part (1) above, as the next theorem shows. If $A$ is 
AS regular with trivial ozone group, then since the Nakayama 
automorphism $\mu_A \in \Oz(A)$, $A$ is necessarily CY. The 
classification of quadratic CY algebras of dimension 3 which are finite 
over their centers is contained in the work of Itaba and Mori 
\cite[Lemma 3.2]{IM1}. This can also be worked out from the 
classification Mori and Smith 
\cite{MS2}. We note that a connected graded PI ring is finite 
over its center \cite[Corollary 1.2]{StaZ}.

\begin{thmintro}[Theorem~\ref{yythm0.6}]
\label{yythmintro0.6}
Let $\kk$ be algebraically closed
and let $A$ be a PI quadratic AS regular algebra of global 
dimension 3. Then $A$ has trivial ozone group if and only if
it is isomorphic to one of the following:
\begin{itemize}
\item 
an elliptic Sklyanin algebra whose defining automorphism has order $n$ where $n \neq 1$ and $3 \nmid n$.
\item 
one of the algebras $B_q$ {\rm{(}}see \eqref{E1.4.1}{\rm{)}} 
where $q \neq 1$ is a root of unity and $3$ does not divide the order of $q$.
\end{itemize}
\end{thmintro}

In Section~\ref{yysec3} we study the relationship between 
a noetherian PI AS regular algebra $A$, the invariant ring 
$A'=A^{\Oz(A)}$, and the skew group ring 
$\overline{A}=A\#\kOz(A)$, as well as their respective centers 
$Z$, $Z'$, and $\overline{Z}$. We have the following diagram:
\[\begin{CD}
A' @> \subset >> A @> \subset >> \overline{A}\\
@A \cup AA @A \cup AA @AA \cup A\\
Z' @<< \supset < Z @>> \subset > \overline{Z}
\end{CD}\]
If $\Oz(A)$ is trivial, then $(\overline{A}, \overline{Z})
=(A',Z')=(A,Z)$. See Theorem~\ref{yythmintro0.6}. In general, 
however, the relationship between the six algebras in the 
above diagram is mysterious. The primary goal of Section~
\ref{yysec3}, however, is to provide the necessary 
tools to prove the following theorem.

\begin{thmintro}[Theorem~\ref{yythm0.7}]
\label{yythmintro0.7}
Let $A$ be a noetherian PI AS regular algebra with center $Z$.
Then the order of $\Oz(A)$ is finite and divides $\rk_Z(A)$.
\end{thmintro}

Theorem~\ref{yythmintro0.6} exhibits one extreme case of equation 
\eqref{E0.0.1}. To understand the other extreme (namely, 
$\left|\Oz(A)\right|=\rk_Z(A)$), we recall the definition of 
a skew polynomial ring. 

Let $\bp := (p_{ij}) \in M_{n \times n}(\kk)$ be a 
multiplicatively antisymmetric matrix. The \emph{skew polynomial ring} $\SP$ is defined to be 
\begin{align}\label{E0.7.1}
\kk_{\bp}[x_1,\dots,x_n]
    = \kk\langle x_1,\hdots,x_n\rangle/(x_j x_i- p_{ij} x_i x_j).
\end{align}
When there exists $q \in \kk^\times$ such that $q=p_{ij}$ for all 
$i<j$ we write $\kk_q[x_1,\hdots,x_n]$.

The ring-theoretic properties of the $\SP$ are well-known. In 
particular, $\SP$ is a noetherian Artin--Schelter regular 
algebra of global and Gelfand--Kirillov dimension $n$. It is 
easy to see that $\SP$ is PI if and only if each $p_{ij}$ is a 
root of unity. In \cite{CGWZ2}, the present authors studied the 
ozone groups of PI skew polynomial rings and their applications. 
In low dimension, this was used to determine conditions on the 
parameters which are equivalent to the regularity and 
Gorensteinness of the center $\ZSP$. We prove the following result 
about skew polynomial rings.

\begin{thmintro}[Theorem~\ref{yythm0.8}/Skew Polynomial Characterization Theorem]
\label{yythmintro0.8}
Suppose $\kk$ is algebraically closed. Let $A$ be a noetherian 
PI AS regular algebra that is generated in degree one. The 
following are equivalent.
\begin{enumerate}
\item[(1)] 
$A$ is isomorphic to $S_{\mathbf p}$.
\item[(2)] 
$\Oz(A)$ is abelian and $\left|\Oz(A)\right| = \rk_Z(A)$.
\item[(3)] 
$A$ is generated by normal elements.
\item[(4)] 
$\Oz(A)$ is abelian and $A^{\Oz(A)}=Z(A)$.
\end{enumerate}
\end{thmintro}

The above theorem is a characterization of skew polynomial 
rings $\SP$ in the graded setting (probably the 
first of this kind). It is unclear to us how to intrinsically 
characterize a skew polynomial ring as an ungraded algebra. 
In his Bourbaki Seminar talk 
\cite{kraft}, 
Kraft listed eight challenging problems for commutative 
polynomial rings $\kk[x_1,\cdots,x_n]$ (mostly for $n\geq 3$). 
Three of the problems listed were the automorphism problem, 
the cancellation problem, and the characterization problem. Similar 
questions can be asked for noncommutative AS regular algebras, and we see that
the ozone group is related to both the automorphism 
problem and the characterization problem in this context. It would be interesting to see 
applications to other types of problems as well. 

Indeed, as an application of Theorem~\ref{yythmintro0.8}, we are also able to study Zariski 
cancellation of skew polynomial rings. We say an algebra 
$A$ is \emph{cancellative} if an isomorphism $A[t] \cong B[t]$ 
implies an isomorphism $A \cong B$. The Zariski cancellation 
problem asks under what conditions an algebra $A$ is 
cancellative. In \cite{BZ2} the authors initiated a study of 
this problem in the context of noncommutative algebras.

Skew polynomial rings with trivial centers are cancellative by 
\cite[Proposition 1.3]{BZ2}. By \cite[Theorem 8]{BHZ2}, a 
skew polynomial ring with single parameter $p\neq 1$ 
in an even number of variables is cancellative. In three 
variables, skew polynomial rings that are not PI are 
cancellative \cite[Theorem 0.2]{TTZ}.

Let $A = \kk_p[x,y,z]$ with $p$ a root of $1$. 
When $p=1$, it is an open problem to determine whether $A$ 
is cancellative in general. If we restrict our attention to 
connected graded algebras which are generated in degree one, 
then $A$ is cancellative by \cite[Theorem 9]{BZ1}. 
Here we use 
ozone groups to remove the hypothesis that $A$ is generated in 
degree one from the above result.

\begin{thmintro}[Corollary~\ref{yycor4.15}]
\label{yythmintro4.14}
If $p$ is a root of unity, then $\kk_p[x,y,z]$ is cancellative in the class of connected graded algebras.
\end{thmintro}

Theorem~\ref{yythmintro4.14} is a corollary of a stronger result, which addresses the multiparameter case $\kk_{\bp}[x,y,z]$ (see Theorem~\ref{yythm4.14}).

As in Theorems~\ref{yythmintro0.3} and \ref{yythmintro0.8}, the rank of the ozone group over the center gives meaningful information about the algebra in two extreme cases 
(either $\Oz(A)$ is the largest possible in terms of $\rk_Z(A)$
or the smallest possible). When $|\Oz(A)|$ is strictly between 
$1$ and $\rk_Z(A)$, see \eqref{E0.0.1}, it should also 
provide useful information about the algebra $A$. On the other 
hand, except for Theorems~\ref{yythmintro0.4} and \ref{yythmintro0.7}, 
very little is known about general properties of the ozone groups. 
For example,

\begin{question}
\label{yyque0.10}
Is the ozone group of a noetherian PI AS regular algebra 
abelian?
\end{question}

In all examples of noetherian PI AS regular algebras we have 
considered so far, the ozone group is abelian. However, most 
of our examples are low-dimensional so it may be possible 
that there are higher-dimensional examples with non-abelian 
ozone groups. Other questions are listed in Section~\ref{yysec5}.

\subsection*{Acknowledgments}
R. Won was partially supported by an AMS--Simons Travel Grant 
and Simons Foundation grant \#961085. J.J. Zhang was partially 
supported by the US National Science Foundation (Nos. DMS-2001015
and DMS-2302087).

\section{Definitions, examples, and extensions}
\label{yysec1}

In this paper we are mainly interested in noetherian PI AS 
regular domains, but the following definition makes sense in 
a more general setting. 

\begin{definition}
\label{yydef1.1}
Let $A$ be an algebra with center $Z$. Let $C$ be a subalgebra 
of $Z$. The \emph{Galois group of $A$ over $C$} is defined as
\[ \Gal(A/C) := \{ \sigma \in \Aut(A) \mid \sigma(c)
=c \text{ for all $c \in C$} \}.\]
When $C=Z$, we call $\Gal(A/Z)$ the \emph{ozone group} of $A$ 
and denote it $\Oz(A)$.

If $A$ is graded and $C$ a graded subalgebra of $Z$, the 
\emph{graded Galois group of $A$ over $C$} is defined as
\[
\Gal_{\gr}(A/C) := \{ \sigma \in \Aut_{\gr}(A) \mid \sigma(c)
=c \text{ for all $c \in C$} \}.
\]
When $C=Z$, we call $\Gal_{\gr}(A/Z)$ the \emph{graded ozone 
group} and denote it $\Oz_{\gr}(A)$.
\end{definition}

Our first example 
is from \cite{CGWZ2}, 
where ozone groups of skew polynomial rings are studied more 
extensively.

\begin{example}
\label{yyexa1.2}
Let $A=\kk_q[x,y]$ where $q \in \kk^\times$ is a primitive 
$n$th root of unity, $n>1$. It is well-known that 
$Z(A)=\kk[x^n,y^n]$. 

For $q\neq -1$, if $\sigma \in \Aut(A)$ then $\sigma(x)=\mu x$ 
and $\sigma(y)=\nu y$ for some $\mu, \nu \in \kk^\times$. It 
follows that if $\sigma \in \Oz(A)$ then $\mu$ and $\nu$ are 
$n$th roots of unity. If $q=-1$ there are additional automorphisms 
given by $\tau(x)=\mu y$ and $\tau(y)=\nu x$ for 
$\mu,\nu \in \kk^\times$. These automorphisms, however, do not 
fix the center. Thus, in either case, we have $\Oz(A) = \cyc_n 
\times \cyc_n$.

Note that, in general, the ozone group of an algebra may depend 
on the base field $\kk$, and is not preserved under extension of 
scalars. For example, let $B=\RR\langle x,y\rangle/(x^2+y^2)$. 
It is easy to check that $\Oz(B)=\{\Id, \phi\} \cong \cyc_2$ 
where $\phi: x\mapsto -x, y\mapsto -y$. However, 
$B\otimes_\RR \CC \cong \CC_{-1}[x,y]$, which we computed 
above has ozone group $\cyc_2 \times \cyc_2$.
So $\Oz(B)\not\cong \Oz(B\otimes_\RR {\CC})$.
\end{example}

In our next set of examples and results, we aim to show that 
every finite abelian group occurs as the ozone group of some 
AS regular algebra.

Below, and throughout, we use the following useful result 
which follows from the proof of \cite[Proposition 1.8]{KWZ1},
\begin{align}\label{E1.2.1}
\rk_Z(A)=(h_A(t)/h_Z(t))\restrict{t=1}.
\end{align}
Here $h_A(t)$ denotes the Hilbert series of $A$ (in the variable $t$).

\subsection{Quantum Heisenberg algebras}
\label{yysec1.1}
Let $q \in \kk^\times$. The 
\emph{quantum Heisenberg algebra}, 
denoted $H_q$, is the $\kk$-algebra
\begin{align}\label{E1.2.2}
H_q := \kk\langle x,y,z \rangle/
(zx - qxz, yz - q zy, xy-qyx - z^2).
\end{align}
The element $\Omega:=xy-q^{-2}yx$ is normal in $H_q$.

By \cite[Theorem 3.3]{IM1}, $H_q$ is PI if and only if 
$q$ is a primitive $\ell$th root of unity where 
$\ell^3 \neq 1$. A variation of the $H_q$, with $z^2$ 
replaced by $z$, was studied by Kirkman and Small 
\cite{KS}. The $H_q$ are also examples of 
\emph{ambiskew polynomial rings} as studied by Jordan 
\cite{Jiter}.

\begin{lemma}
\label{yylem1.3}
Let $q \in \kk^\times$ be a primitive $\ell$th root of 
unity where $\ell\geq 2$, $\ell \neq 3$. Set
\[ \Omega := xy-q^{-2}yx.\]
Then $Z(H_q)$ is generated by $x^\ell$, $y^\ell$, 
$z^\ell$, and $\Omega z$.
\end{lemma}

\begin{proof}
Let $\eta_z:H_q \to H_q$ be the automorphism of $H_q$ 
given by $\eta_z(a)=z\inv az$. Then an application of 
Molien's Theorem \cite[Lemma 5.2]{JiZ} shows that the 
fixed ring $\widehat{H}=H_q^{\grp{\eta_z}}$ is generated 
by $x^\ell$, $y^\ell$, $yx$, and $z$. Alternatively, we 
can replace $yx$ by $\Omega$ in the generating set. 

The map $\eta_x:\widehat{H} \to \widehat{H}$ given by 
$\eta_x(a)=x\inv ax$ is an automorphism. Clearly, 
$\eta_x(x^\ell)=x^\ell$, $\eta_x(z)=qz$, and 
$\eta_x(\Omega)=q\inv\Omega$. Induction shows that 
$\eta_x(y^\ell)=y^\ell$. Then it is easy to see that 
$\widehat{H}^{\grp{\eta_x}}$ is generated by $x^\ell$, 
$y^\ell$, $z^\ell$, and $\Omega z$. These elements are 
obviously fixed by conjugation by $y$, giving the result.
\end{proof}

\begin{proposition}
\label{yypro1.4}
Let $q \in \kk^\times$ be a primitive $\ell$th root of 
unity where $\ell\geq 2$, $\ell \neq 3$.
\begin{itemize}
\item[(1)]
If $3 \nmid \ell$, then $\Oz(H_q) \iso \cyc_\ell$.
\item[(2)] 
If $3 \mid \ell$, then $\Oz(H_q) \iso \cyc_\ell \times \cyc_3$.
\end{itemize}
\end{proposition}

\begin{proof}
Let $\phi \in \Oz(H_q)$. A similar computation as in 
Example~\ref{yyexa1.2} using the defining relations shows 
that $\phi(x)=\epsilon_1 x$, $\phi(y)=\epsilon_2 y$, and 
$\phi(z)=\epsilon_3 z$ for some $\epsilon_i \in \kk^\times$.

By Lemma~\ref{yylem1.3}, $Z(H_q)$ is generated by $x^\ell$, 
$y^\ell$, $z^\ell$, and $\Omega z$. Thus, the $\epsilon_i$ 
are $\ell$th roots of unity. We have
\[
\Omega z = \phi(\Omega z) = 
(\epsilon_1 \epsilon_2 \epsilon_3) \Omega z.
\]
Moreover,
\[
0 = \phi(xy-qyx-z^2) = 
\epsilon_1\epsilon_2(xy-qyx)-\epsilon_3^2 z^2 = 
(\epsilon_1\epsilon_2-\epsilon_3^2)z^2.
\]
Combining these gives $\epsilon_3\inv = \epsilon_1\epsilon_2 
= \epsilon_3^2$, so $\epsilon_3^3=1$. 

If $3 \nmid \ell$, then $\epsilon_3=1$ and so 
$\epsilon_2=\epsilon_1\inv$. So, in this case, 
$\Oz(H_q) \iso \cyc_\ell$.

If $3 \mid \ell$, then $\epsilon_3=\omega$ for some 
primitive third root of unity $\omega$. Then 
$\epsilon_1 \in \cyc_\ell$ and $\epsilon_2
=\epsilon_3^2\epsilon_1\inv \in \cyc_\ell$. Thus, 
$\Oz(H_q) \iso \cyc_\ell \times \cyc_3$.
\end{proof}

\begin{remark}\label{rem.hberg}
For $q \in \kk^\times$, consider the $\kk$-algebras
\[H_q' := \kk\langle x,y,z \rangle/(xz-qzx, zy-qyz, xy-qyx-z^2).\]
These are the non-CY quantum Heisenberg algebras.

Let $q$ be an $\ell$th root of unity, where $\ell\geq 3$. Set $\Omega=xy-q^2yx$.
A similar computation to above shows that the center of $H_q'$ is generated by $x^\ell$, $y^\ell$, $z^\ell$, and $\Omega z^{\ell-1}$.
Let $\phi \in \Oz(H_q')$, then as above,
$\phi(x)=\epsilon_1 x$, $\phi(y)=\epsilon_2 y$, and $\phi(z)=\epsilon_3 z$ where each $\epsilon_i$ is an $\ell$th root of unity. We have
\[
\Omega z^{\ell-1} 
    = \phi(\Omega z^{\ell-1})
    = (\epsilon_1\epsilon_2 \epsilon_3\inv)
    \Omega z^{\ell-1}. 
\]
Moreover,
\[
0 = \phi( xy-qyx-z^2 ) = (\epsilon_1\epsilon_2)(xy-qyx)- \epsilon_3^2z^2
= (\epsilon_1\epsilon_2 - \epsilon_3^2)z^2.
\]
Combining these gives, $\epsilon_3 = \epsilon_1\epsilon_2 = \epsilon_3^2$,
so $\epsilon_3=1$. It then follows that $\epsilon_2=\epsilon_1\inv$. This implies that $\Oz(H_q')=\langle \phi \rangle \cong \cyc_\ell$.
\end{remark}

\subsection{\texorpdfstring{The algebras $B_q$}{The algebras B}}
For $q \in \kk^\times$, $q^3 \neq 1$, we define
\begin{align}\label{E1.4.1}
B_q := \kk\langle x,y,z \rangle/
(xy-qyx, zx-qxz-y^2, zy-q\inv yz - x^2 ).
\end{align}

These algebras are example of quadratic CY algebras of global 
dimension three \cite{IM1,MS2}. In particular, after a change of basis, 
these are the 3-dimensional quadratic AS regular algebras of type 
$\mathrm{NC}$ (their geometric type is a nodal cubic).  In particular, 
$B_q$ is a finite module over its center if and only if $q$ is a root 
of unity. The algebras $B_q$ also appear in \cite[Section 2]{NVZ}.

Like the Skylanin algebras in Section~\ref{yysec2}, certain 
$B_q$ have trivial ozone groups. The following is a special case but 
illustrates the behavior of ozone groups for \emph{most} PI $B_q$.

\begin{example}
\label{yyexa1.5}
Set $B=B_{-1}$. It is easy to verify that, in this case, 
$x^2$, $y^2$, $z^2$, and $\Omega=2xyz+y^3-x^3$ are all central 
(see also Lemma~\ref{yylem1.6} below). We show that $B$ has 
trivial ozone group.

Let $\phi \in \Oz(B)$. Write $\phi(x)=ax+by+cz$. Thus,
\[ x^2 = \phi(x^2) = (ax+by+cz)^2 
= (a^2-bc)x^2 + (b^2-ac)y^2 + c^2z^2.\]
The coefficient of $z^2$ shows that $c=0$. But then the coefficient 
of $y^2$ shows that $b=0$. Hence, $\phi(x)=\pm x$. An identical 
computation shows that $\phi(y)=\pm y$. 

Write $\phi(z)=\alpha x + \beta y + \gamma z$. Then
\begin{align*}
0 = \phi(zy+yz+x^2) 
	= \pm \left( 2\beta y^2 + \gamma (zy+yz) \right) + x^2,
\end{align*}
so $\beta=0$. A similar computation with the last defining 
relation shows that $\alpha=0$. Hence, $\phi(z)=\gamma z$ so 
$\phi(z)=\pm z$.

Since $\phi$ scales all of the generators, then we require 
$x^3=\phi(x^3)$, so $\phi(x)=x$. Similarly, $\phi(y)=y$. Since 
$\phi(\Omega)=\Omega$, then we have $\phi(xyz)=xyz$. Thus, 
$\phi(z)=z$. That is, $\phi$ is the identity.
\end{example}

We now consider the general PI case.

\begin{lemma}
\label{yylem1.6}
Suppose that $q$ is a primitive $n$th root of unity, $n\neq 1,3$. 
Set $\alpha = (q^{-2}-q)\inv$, $\beta=(q^2-q\inv)\inv$, and 
\[ d = \alpha x\inv y^2 + \beta y\inv x^2.\]
Let $z'=z-d$. Then $Z(B_q)$ is generated by $x^n$, $y^n$, $z^n$, 
and $xyz'$.
\end{lemma}

\begin{proof}
We observe that $B_q$ is the Ore extension 
$\kk_q[x,y][z;\sigma,\delta]$ where
\[
\sigma(x)=qx, \quad \sigma(y)=q\inv y, 
\qquad \delta(x)=y^2, \quad \delta(y)=x^2.
\]
Let $X$ be the Ore set of $B_q$ generated by powers of $x$ and 
$y$, and set $\widehat{B_q} = B_qX\inv$. Then
$\widehat{B_q} = \kk_q[x^{\pm 1},y^{\pm 1}][z;\sigma,\delta]$ 
where $\sigma$ and $\delta$ are extended in the natural way.

We have
\begin{align*}
dx-\sigma(x)d &= \alpha (q^{-2} - q)y^2 + \beta 
 ( xy\inv -q y\inv x) x^2 = y^2 = \delta(x) \\
dy-\sigma(y)d &= \alpha (x\inv y - q\inv yx\inv) y^2 
 + \beta (q^2 - q\inv)x^2 = x^2 = \delta(y).
\end{align*}
Thus, $\delta$ is an inner $\sigma$-derivation. Set $z'=z-d$. Then
\[ \widehat{B_q} = \kk_q[x^{\pm 1},y^{\pm 1}][z';\sigma].\]

Let $S=\kk_\bp[x_1,x_2,x_3]$ be the skew polynomial ring with 
parameters $p_{12}=q\inv$, $p_{13} = q$, $p_{23}=q\inv$. Let 
$Y$ be the Ore set of $S$ generated by powers of $x_1$ and $x_2$, 
and set $\widehat{S}=SY\inv$. Then it is clear that 
$\widehat{S} \iso \widehat{B_q}$.

Note that $S$ is CY and $Z(S)$ is generated by $x_1^n$, $x_2^n$, 
$x_3^n$, and $x_1x_2x_3$. Hence, $Z(\widehat{S})$ is generated 
by the same elements along with appropriate inverses. Since 
$Z(B_q) = Z(\widehat{B_q}) \cap B_q$, then we need only verify 
that the generators in the statement belong to $B_q$. This is 
clear for $x^n$, $y^n$, and $xyz'$. Note that
\[ (xyz')^n = q^{-s_{n-1}} x^ny^n(z')^n\]
where $s_n$ is the sum of the first $n$ positive integers. 
Since $x^n, y^n, (xyz')^n \in Z(B_q)$, then a degree argument 
implies that $(z')^n \in Z(B_q)$ as well. Now an induction on 
powers of $z'$ shows that $z^n \in Z(B_q)$.
\end{proof}

We now set
\[\Omega = xyz' = 
xyz - q(q^{-2}-q)\inv y^3 - (q^2-q\inv)\inv x^3 \in Z(B_q).\]

\begin{proposition}
\label{yypro1.7}
Let $B=B_q$ where $q$ is a root of unity of order $n$ for 
$n \neq 1,3$.
\begin{enumerate}
\item[(1)] 
If $3 \nmid n$, then $\Oz(B)$ is trivial.
\item[(2)] 
If $3 \mid n$, then $\Oz(B)=(\cyc_3)^2$.
\end{enumerate}
\end{proposition}

\begin{proof}
The case $n=2$ is explained in Example~\ref{yyexa1.5}. We assume 
$n>2$ and $n \neq 3$. By Lemma~\ref{yylem1.6}, $x^n$, $y^n$, and 
$\Omega$ are central. Let $\phi \in \Autgr(A)$. A straightforward 
computation using the first defining relation shows that 
$\phi(x),\phi(y) \in \kk x + \kk y$. Because we have assumed 
$q \neq \pm 1$, then we have $\phi(x) \in \kk x$ and $\phi(y) 
\in \kk y$.

Write $\phi(x)=a_1x$, $\phi(y)=b_2y$, and 
$\phi(z)=a_3x+b_3y+c_3z$. Applying $\phi$ to the second 
defining relation gives
\[
0 = ((1-q^2)q\inv a_1b_3) xy + ((1-q)a_1a_3)x^2 + (a_1c_3-b_2^2)y^2.
\]
Since $a_1 \neq 0$, this implies that $a_3=b_3=0$. Thus, 
$\phi(z) \in \kk z$.

Now suppose that $\phi \in \Oz(A)$ and suppose that $q$ has order 
$n$, $n > 3$. Write $\phi(x)=ax$, $\phi(y)=by$, and $\phi(z)=cz$.
Then $x^n = \phi(x^n)$, so $a$ is an $n$th root of unity. The 
same holds for $b$. Since $\Omega=\phi(\Omega)$ and the 
automorphisms are scalar, this implies that 
$x^3=\phi(x^3) = a^3x^3$ and $y^3=\phi(y^3) = b^3y^3$. Moreover, 
$xyz=\phi(xyz)=abc (xyz)$

Now if $3 \nmid n$, we have that $a=b=c$. On the other hand, if 
$3 \mid n$, then we may take $a$ and $b$ to be (arbitrary) 
primitive third roots of unity and $c=(ab)\inv$. One verifies 
easily that any such automorphism scales $z'$ by $c$ and so 
fixes $(z')^n$. Thus, $\Oz(B_q)=(\cyc_3)^2$ as claimed.
\end{proof}

Further examples are given below. The ozone groups of 
three-dimensional PI Sklyanin algebras are computed in 
Section~\ref{yysec2}. Except for some exceptional 
examples, the ozone groups of these Sklyanin algebras are 
trivial. In Example~\ref{yyexa4.8} below we compute the ozone 
group for some PI down-up algebras. These examples combined 
suggest that the ozone group of a PI AS regular algebra is 
abelian (see Question~\ref{yyque0.10}). We do not know if this 
is true in general or an artifact of our choice of examples. 

\subsection{Ozone groups of extensions}
\label{yysec1.3}
For the remainder of this section we present some general 
results regarding ozone groups of PI AS regular algebras. 
In particular, we provide the last piece for the proof of 
Theorem~\ref{yythmintro0.4} and show how ozone groups behave 
under certain Ore extensions. 

The next lemma will be important in characterizing skew 
polynomial rings as graded algebras with ozone groups of 
maximal rank.

If $w$ is a regular normal element of an algebra $A$, let
$\eta_w: A \to A$ be the map defined by $\eta_w(x) = w^{-1} xw$ 
for all $x\in A$. Then $\eta_w$ is an algebra automorphism 
of $A$ (if $w$ is invertible, it is the conjugation 
automorphism).

\begin{lemma}
\label{yylem1.9}
Suppose $A$ is a prime algebra which is a finite module 
over its center $Z$. Let $\phi\in \Oz(A)$. 
\begin{enumerate}
\item[(1)]
There is a regular normal element $a\in A$ such that 
$\phi = \eta_a$, that is, $\phi(x)= a^{-1} x a$ for all $x\in A$.
\item[(2)]
If further $A$ is a $G$-graded domain where $G$ is an ordered 
abelian group, then there is a nonzero normal homogeneous 
element $a\in A$ such that $\phi=\eta_a$. 
\item[(3)]
If $A$ is a $\ZZ^n$-graded domain, then $\Oz_{\gr}(A)=\Oz(A)$.
\end{enumerate}
\end{lemma}

\begin{proof} (1) Let $F:=Z({Z\setminus \{0\}})^{-1}$ be the 
field of fractions of $Z$. Since $A$ is prime, the localization 
$B := A({Z\setminus \{0\}})^{-1}$ is a central simple algebra 
with center $F$. Using localization, $\phi$ extends to an 
automorphism of $B$, which we denote by $\phi$, by abuse of 
notation. Since $\phi$ fixes $F$, by the Skolem--Noether theorem, 
$\phi$ is an inner automorphism of $B$. Therefore there is an 
invertible element $f\in B$ such that $\phi=\eta_f$. Write 
$f=az^{-1}$ where $a\in A$ is regular and $z\in Z$ is nonzero. 
Since $z$ is central, we have $\phi(x)=\eta_f(x)=\eta_a(x)$ 
for all $x\in A$. Thus we obtain that 
\begin{equation}\label{E1.9.1}
a \phi(x)=xa
\end{equation}
for all $x\in A$. Since $\phi$ is an automorphism of $A$, the
above equation implies that $a$ is normal. So the assertion 
follows.

(2) Consider \eqref{E1.9.1} for a homogeneous element $x \in A$.
Since $G$ is an ordered group and $A$ is a domain, 
\eqref{E1.9.1} implies that $\phi(x)$ is also homogeneous 
and $\deg_G \phi(x)=\deg_G x$. If we write $a$ as 
$a_1+\cdots+a_n$ where the $a_i$ are nonzero homogeneous 
components of $a$ of different $G$-degrees, then \eqref{E1.9.1} 
implies that $a_i \phi(x)=xa_i$ for each $i$. This means that 
each $a_i$ is a normal homogeneous element in $A$ and 
$\eta_{a_i}=\eta_a$ for $i=1,\cdots,n$. So $\phi=\eta_{a_1}$.

(3) This  follows from part (2).
\end{proof}

There are many consequences of Lemma~\ref{yylem1.9}. Most 
importantly, if $A$ is a graded domain which is module finite 
over its center, then we need not distinguish between
$\Oz_{\gr}(A)$ and $\Oz(A)$. If $A$ is a PI AS regular algebra, 
then since the Nakayama automorphism $\mu_A$ is in $\Oz(A)$, we 
see that $\mu_A=\eta_f$ for a regular normal homogeneous element 
$f \in A$. Some additional consequences are given below.

\begin{lemma}
\label{yylem1.10}
Let $A$ be a $\ZZ$-graded domain which is a finite module over 
its center. Then $\Oz(A)$ is trivial if and only if every normal 
element is central.
\end{lemma}

\begin{proof}
Suppose that $A$ is a $\ZZ$-graded domain which is a finite 
module over its center. Let $N^\ast$ (resp. $N_{\gr}^\ast$) be 
the semigroup of all nonzero (resp. nonzero and homogeneous) 
normal elements in $A$. Let $Z^\ast$ (resp. $Z_{\gr}^\ast$) be 
the semigroup of all nonzero (resp. nonzero and homogeneous) 
central elements in $A$. Then there is a map 
$\Xi: N^\ast\to \Oz(A)$ 
(resp. $\Xi_{\gr}: N_{\gr}^\ast\to \Oz_{\gr}(A)$) 
sending an element $a$ to $\eta_a$. This map is surjective by 
Lemma~\ref{yylem1.9}(1, 2). It is easy to see that $\Xi$ and 
$\Xi_{\gr}$ are morphisms of semigroups and map $Z^{\ast}$ and 
$Z_{\gr}^{\ast}$ to the identity. We define an equivalence 
relation on $N^\ast$ (resp. $N_{\gr}^{\ast}$) generated by the 
relation
\begin{align}\label{E1.10.1}
a\sim b \Leftrightarrow z_1 a= z_2 b
\end{align}
for some $z_1,z_2\in Z^{\ast}$ 
(resp. $z_1, z_2\in Z_{\gr}^{\ast}$). It is 
clear that $(N^{\ast}/\sim)\cong \Oz(A)$ (resp. 
$(N_{\gr}^{\ast}/\sim) \cong \Oz_{\gr}(A)$). So we can consider 
$Z^{\ast}$ (resp. $Z_{\gr}^\ast$) as the ``kernel'' of $\Xi$ 
(resp. $\Xi_{\gr}$). From this, one sees that $\Oz(A)$ is 
trivial if and only if every normal element is central. Also 
by the proof of Lemma~\ref{yylem1.9}(2), every regular normal 
element is a sum of homogeneous normal elements in the same 
equivalence class.
\end{proof}

\begin{theorem}[Theorem~\ref{yythmintro0.3}]
\label{yythm0.3}
Consider the following conditions for a noetherian
PI AS regular algebra $A$.
\begin{enumerate}
\item[(1)] $\Oz(A)=\{1\}$. 
\item[(2)] Every normal element of $A$ is central.
\item[(3)] $A$ is CY.
\item[(4)] $Z$ is Gorenstein.
\end{enumerate}
Then $(1) \Leftrightarrow (2) \Rightarrow (3) \Rightarrow (4)$.
\end{theorem}
\begin{proof}
The statement $(3) \Rightarrow (4)$ follows from 
\cite[Lemma 2.5(5)]{SVdB}.  That $(1) \Rightarrow (3)$ follows 
from the fact that the Nakayama automorphism belongs to 
$\Oz(A)$ \cite[Proposition 4.4]{BZ}. Finally 
$(1) \Leftrightarrow (2)$ is Lemma~\ref{yylem1.10}.
\end{proof}

Modulo a small detail, which is completed in the next section, 
the following completes the proof of Theorem~\ref{yythmintro0.4}.

\begin{lemma}
\label{yylem1.11}
Let $A$ and $B$ be noetherian PI AS regular domains.
Then $\Oz(A \otimes B) = \Oz(A) \times \Oz(B)$.
In particular, $\Oz(A[t]) = \Oz(A)$.
\end{lemma}

\begin{proof}
The inclusion $\supset$ is clear. 
For the reverse inclusion, note that since $A$ and $B$ are PI, they are finite over their affine centers. Hence, $Z(A \otimes B) = Z(A) \otimes Z(B)$ is affine and $A \otimes B$ is finite over $Z(A \otimes B)$. Therefore $A$ is noetherian PI.
By the K\"{u}nneth formula, $\gldim(A \otimes B) = \gldim(A) + \gldim(B)$ and hence $A \otimes B$ is AS regular. 
Since $A \otimes B$ is a noetherian PI AS regular algebra, it is a domain \cite{StaZ}.
By Lemma~\ref{yylem1.9}(3), 
$\Oz_{\gr}(A \otimes B)=\Oz(A \otimes B)$. Let 
$\phi \in \Oz(A \otimes B)$ and let $a$ be any homogeneous 
element of $A$. Then $\phi(a \otimes 1) \in A \otimes B_0 = 
A \otimes \kk$. Consequently, $\phi$ restricts to an automorphism 
of $A$. By the same reasoning, $\phi$ restricts to an 
automorphism of $B$. Thus, $\phi \in \Oz(A) \times \Oz(B)$.
\end{proof}

\begin{theorem}[Theorem~\ref{yythmintro0.4}]
\label{yythm0.4}
Let $G$ be a finite abelian group. Then there exists a 
noetherian PI AS regular algebra $A$ such that $\Oz(A)=G$.
\end{theorem}
\begin{proof} 
For cyclic groups $\cyc_\ell$, we can apply Proposition 
\ref{yypro1.4} 
and Remark~\ref{rem.hberg}. Now 
$\Oz(A \otimes B) = \Oz(A) \times \Oz(B)$ by Lemma 
\ref{yylem1.11} and this completes the proof.
\end{proof}

We can generalize the result on $\Oz(A[t])$ to certain Ore extensions $A[t;\sigma]$.

\begin{lemma}
\label{yylem1.12}
Let $A$ be a noetherian PI domain and suppose 
$\sigma \in \Aut(A)$ has finite order $n$. Suppose further 
that $\sigma^i$ is not inner for any $1 \leq i < n$. Then
\[
\Oz(A[t; \sigma]) \cong \cyc_n \times 
\{ \tau \in \Aut(A) \mid \sigma\tau = \tau\sigma 
\text{ and } \tau(a) 
= a \text{ for all $a \in Z(A)^\grp{\sigma}$}\}
\]
where $\cyc_n$ is the cyclic group of automorphisms 
generated by an automorphism which fixes $A$ and maps 
$t$ to $\xi t$ for a primitive $n$th root of unity $\xi$.
\end{lemma}

\begin{proof}
Let $Z = Z(A)$. By our hypotheses on $\sigma$, 
\cite[Proposition 2.3]{matczuk} implies that the 
center of $B=A[t;\sigma]$ is $Z(B) = Z^{\grp{\sigma}}[t^n]$.
(See \cite[Lemma 2.2]{GKM} for a more general result.)

Let $\phi \in \Oz(B)$. We regard $B$ as a $\ZZ$-graded 
algebra with $\deg t = 1$ and $\deg a = 0$ for all $a \in A$. 
Then a simple degree argument shows that $\phi(t)=at + b$ for 
some $a,b \in A$, $a \neq 0$. Since $t^n = \phi(t^n)=\phi(t)^n$ 
then it is easy to see that $b=0$ and $a=\xi^i$ for some 
$0 \leq i <n$.

Let $z \in Z(B)$ and write $z = a_0 + a_1 t^n + a_2 t^{2n} 
+ \cdots + a_k t^{nk}$. Since $z=\phi(z)$, then it follows 
that $\phi(a_i)=a_i$ for each $i$. Thus, $\phi$ restricts 
to an automorphism $\tau$ in $\Gal(A/Z(A)^{\grp{\sigma}})$. 
On the other hand, since $\phi \in \Aut(B)$, then we have
\[
0=\phi(ta-\sigma(a)t) = \xi(t \phi(a) - \phi(\sigma(a)) t) 
= \xi (\sigma\tau-\tau\sigma)(a)t.
\]
Since $A$ is a domain, so is $B$ and hence 
$\sigma\tau = \tau\sigma$.
\end{proof}

\section{Ozone groups of PI Sklyanin algebras in dimension 
three}\label{yysec2}

In this section, we assume that $\kk$ is algebraically closed.
Let $A = S(a,b,c)$ be a PI three-dimensional Sklyanin algebra as defined in \eqref{E0.4.1}. 
The main result in this section is Theorem \ref{yythm0.5}, which shows that, generically, $\Oz(A)$ is trivial.
First, we consider some exceptional cases.

\begin{example}
\label{yyexa2.1}
Let $A=S(1,1,-1)$. Then $A$ is PI and the center of $A$ is 
generated by $t_1:=x^2, t_2:=y^2, t_3:=z^2$, and 
$g:=x^3-y^3-xyz+yxz$. By a computation, 
\[ t_1^3+t_2^3+t_3^3-5 t_1 t_2 t_3-g^2=0.\]
So the center of $A$ is 
\[\kk[t_1,t_2,t_3,g]/(t_1^3+t_2^3+t_3^3-5 t_1 t_2 t_3-g^2).\]
Note that the center has an isolated singularity at the origin.
It follows from \eqref{E1.2.1} that $\rk_Z(A)=4$.
 
Next we will compute the ozone group. Let $\phi\in \Oz(A)$. By Lemma~\ref{yylem1.9}, $\phi$ is graded so suppose
$\phi(x)=ax+by+cz$ where $a,b,c\in \kk$. Then 
\begin{align*}
x^2&= \phi(x)^2=(ax+by+cz)^2\\
&=a^2 x^2 +b^2 y^2+ c^2 z^2+ab(xy+yx)+ac(xz+zx)+bc(yz+zy)\\
&=(a^2+bc)x^2+(b^2+ac)y^2+(c^2+ab)z^2
\end{align*}
which implies that $1=a^2+bc$, $0=b^2+ac$ and $0=c^2+ab$.
So either $a=\pm 1, b=c=0$ or $a=1/\sqrt{2}$, 
$b=\xi_1^{-1}/\sqrt{2}$ and $c=\xi_1/\sqrt{2}$ where 
$\xi_1^3=-1$. We have 3 possibilities
$$\phi(x)=x, \quad {\text{or}} \quad \phi(x)=-x, \quad
{\text{or}} \quad \phi(x)
=1/\sqrt{2} x+ \xi_1^{-1}/\sqrt{2} y+ \xi_1/\sqrt{2}z$$
where $\xi_1^3=-1$. Similarly, we have
$$\phi(y)=y, \quad {\text{or}} \quad \phi(y)=-y, \quad
{\text{or}} \quad \phi(y)
=\xi_2/\sqrt{2} x+ 1/\sqrt{2} y+ \xi_2^{-1}/\sqrt{2}z$$
where $\xi_2^3=-1$. Checking case-by-case, we obtain that either 
$\phi(x)=x$, $\phi(y)=y$ or $\phi(x)=-x$, $\phi(y)=-y$. Similarly, 
we can extend the argument to the $z$ component. So either 
$\phi$ is the identity or $\phi$ sends $f$ to $-f$ if $\deg f=1$. 
Since $\phi$ also preserves the central element $g$, $\phi$ 
must be the identity. 
\end{example}

The case when $|\sigma|=2$ is considered above Example 
\ref{yyexa2.1}. By \cite[Proposition 5.2]{W}, the Sklyanin 
algebras corresponding to automorphisms of order 3 are
\begin{itemize}
\item 
$S(1,0,\omega)$ for $\omega= -1, e^{\pi i/3}, e^{5\pi i/3}$, and
\item 
$S(1,\omega,0)$ for $\omega= e^{\pi i/3}, e^{5\pi i/3}$.
\end{itemize}

\begin{lemma}\label{lem.skpoly}
Let $\SP=\kk_{\bp}[x,y,z]$ be a CY PI skew polynomial ring where the parameters $p_{ij}$ are $n$th roots of unity. Then $\Oz(\SP)=(\ZZ_n)^2$.
\end{lemma}
\begin{proof}
By \cite[Theorem 0.10]{CGWZ2}, $\SP$
is Calabi--Yau if and only the relations are given by
\[ yx = \xi xy, \quad zy = \xi yz, \quad xz = \xi zx.\]
for some $n$th root of unity $\xi$. The ozone group is generated by $\eta_x$, $\eta_y$, and $\eta_z$.
Since $\eta_z = (\eta_x \eta_y)\inv$, we have that $\Oz(\SP) = \langle \eta_x, \eta_y \rangle$, and the result follows.
\end{proof}

\begin{lemma}\label{yylem2.2}
The algebras $S(1,e^{\pi i/3},0)$ and $S(1,e^{5\pi i/3},0)$ are 
skew polynomial rings with ozone groups $(\cyc_6)^2$.
The algebra $S(1,0,-1)$ is isomorphic to a skew polynomial ring with 
ozone group $(\cyc_3)^2$.
\end{lemma}

\begin{proof}
The first two cases follow from Lemma~\ref{lem.skpoly}. For the third case, let $\xi$ be a primitive 
third root of unity. Set
\[ X = x+y+z, \quad Y = x+\xi y+\xi^2 z, \quad Z 
= x+\xi^2 y+ \xi z.\]
It is then easy to verify that
\[ YX = \xi XY, \quad ZY = \xi YZ, \quad XZ 
= \xi ZX.\]
The result now follows from Lemma~\ref{lem.skpoly}.
\end{proof}

The remaining two cases (where $A = S(1,\omega,0)$ for $\omega= e^{\pi i/3}, e^{5\pi i/3}$) require more work.

\begin{lemma}
\label{yylem2.3}
Let $A=S(1,0,e^{\pi i/3})$ or $S(1,0,e^{5\pi i/3})$. 
Then $\Oz(A)=\cyc_3$.
\end{lemma}

\begin{proof}
Suppose $A = S(1, 0, \omega)$ where $\omega = e^{\pi i /3}$ or $\omega = e^{5 \pi i/ 3}$. First we compute some central elements. Observe that in $A$,
\[ x^3 = x(-\omega\inv yz) = -\omega\inv (xy)z = z^3 
= z(-\omega\inv xy) = -\omega\inv (zx)y = y^3.\]
Thus, it is clear that $x^3 \in Z(A)$. Now consider the element
$\Omega_1 = xzy+yxz+zyx$ and note that
\begin{align*}
x\Omega_1
&= (-\omega\inv yz)zy+(-\omega z^2)xz + (xzy)x \\
&= -\omega\inv y(-\omega\inv xy)y -\omega z(-\omega y^2)z + (xzy)x \\
&= \omega^{-2} yx(-\omega\inv zx) +\omega^2 zy(-\omega x^2) + (xzy)x \\
&= \Omega_1 x.
\end{align*}
Since there is an order three automorphism that permutes the 
generators $x \mapsto y \mapsto z \mapsto x$, this proves that 
$\Omega_1$ is indeed central.

Set $\xi=-\omega$, so $\xi$ is a primitive third root of unity.
Finally, we claim that $\Omega_2 = (x^2y+x^2z)+\xi (z^2x + y^2x) 
+ \xi^2 (z^2y + y^2z)$ is central:
\begin{align*}
x\Omega_2
&= (yx^2)x+(zx^2)x + \xi xz^2x + \xi xy^2x + \xi^2 xz^2y + \xi^2 xy^2z  \\
&= \left( -\omega\inv y^2 z -\omega y^2x - \omega\inv \xi  x^2 y 
    -\omega \xi z^2y 
+ \omega^{-2} \xi^2 x^2z + \omega^2 \xi^2 z^2x  \right)x = \Omega_2 x.
\end{align*}
Taking cyclic permutations proves that the given element is central.

Next we compute the ozone group. Let $\phi \in \Oz(A)$ and write 
$\phi(x)=ax+by+cz$ for some scalars $a,b,c \in \kk$ not all zero.
A computation shows that
\begin{align*}
x^3 &= \phi(x^3) = (ax+by+cz)\phi(x^2) \\
&= \left( a^3+b^3+c^3 + 3\xi abc\right)x^3 + (abc)\Omega_1 \\
&\quad + \left( a^2b  + \xi b^2c  + \xi^2 ac^2\right) x^2y 
		+ \left( a^2c + \xi bc^2 + \xi^2 ab^2\right)x^2z \\
&\quad + \left( \xi (ac^2) + b(a^2 +\xi bc)\xi^2\right) y^2z	
		+  \left( ab^2 + c(a^2 +\xi bc)\xi \right) y^2x \\
&\quad + \xi (a^2b + c(b^2-\omega ac)\xi^2)z^2x 	
		+ \left(a(b^2+\xi ac)\xi + bc^2 \right)z^2y.
\end{align*}
The coefficient of $\Omega_1$ forces $abc=0$. Thus, one of 
$a,b,c$ is zero. The remaining relations then require one of 
the other to be zero. That $\phi$ must preserve the relations 
on $S$ shows that $\phi$ is a cyclic permutation which scales 
all generators by the same third root of unity. On the other 
hand, a cyclic permutation does not fix $\Omega_2$. The result 
now follows from \cite[Theorem]{AST}.
\end{proof}

Before moving to the general case, there is one more
special class of Sklyanin algebras which we must consider.

Consider the family of Skylanin algebras $S_\alpha=S(0,1,-\alpha)$ with $\alpha \in \kk^\times$ and $\alpha^3 \neq 1$.
Explicitly, these are generated by $x,y,z$ with relations
\begin{equation} \label{S3}
yx-\alpha z^2, \qquad xz-\alpha y^2, \qquad zy-\alpha x^2.
\end{equation}
These algebras occur as type $\text{S}_3$ in Itaba and Mori's classification of three-dimensional quadratic AS regular algebras \cite{IM1}. The point scheme $E$ for these algebras is a triangle. The automorphism of $E$ is also given in \cite{IM1},
and $\sigma$ has finite order (equivalently, $S_\alpha$ is PI) if and only if $\alpha$ is a primitive $m$th root of unity. In this case, $|\sigma| = \operatorname{lcm}(3,m)$.

Let $\tau$ be the (graded) automorphism of $S_{\alpha}$ defined by $x \mapsto z \mapsto y \mapsto x$. We observe that the graded twist of $S_\alpha$ by $\tau$ is a skew polynomial ring. 
Let $\omega$ be a primitive third root of unity and let $\rho$ be the automorphism of $S_{\alpha}$ which scales $x$, $y$, and $z$ by $\omega$.

\begin{lemma}\label{lem.special1}
Suppose that $\alpha$ is a primitive $m$th root of unity. For each $k \geq 1$, define $\theta_k = \underbrace{xyzxyz\cdots \tau^{-(k-1)}(x)}_{\text{$k$ factors}} \in S_{\alpha}$. Further define
\[ \Omega = \theta_m + \tau\inv(\theta_m) + \tau^{-2}(\theta_m) \quad \text{and} \quad \Phi = \theta_m + \omega\tau\inv(\theta_m) + \omega^2\tau^{-2}(\theta_m).\]
\begin{itemize}
\item If $3 \nmid m$, then $\Omega$ is a normal element which induces the automorphism $\tau^{-m}$. 
Further, $\Phi$ is a normal element which induces the automorphism $\rho \circ \tau^{-m}$. 
In particular, $\grp{\tau, \rho} \leq \Oz(S_{\alpha})$.

\item If $3 \mid m$, then $\Phi$ is a normal element which induces the automorphism $\rho$. In particular, $\rho \in \Oz(S_{\alpha})$.
\end{itemize}
In either case, $\Oz(S_{\alpha})$ is nontrivial.
\end{lemma}
\begin{proof}
First, observe that
\[ x \tau\inv(\theta_m) = \theta_m \tau^{-m}(x).\]
Now we claim that for any $k \geq 1$, we have
\[
x \tau^{-2}(\theta_k)
	= \alpha^k \tau\inv(\theta_k) \tau^{-k}(x).
\]
If $k = 1$, then the result holds since $x\tau^{-2}(x) = xz = \alpha y^2 = \alpha \tau\inv(x)\tau\inv(x)$ is one of the relations of $S_{\alpha}$. So assume the result holds for some $k \geq 1$. Then
\begin{align*}
x\tau^{-2}(\theta_{k+1}) &= x \tau^{-2}(\theta_k)\tau^{-2-k}(x) = \alpha^k\tau\inv(\theta_k) \tau^{-k}(x)\tau^{-2-k}(x) \\
&= \alpha^k\tau\inv(\theta_k) \tau^{-k}(xz) = \alpha^k\tau\inv(\theta_k) \tau^{-k}(\alpha y^2) \\
&= \alpha^{k+1}\tau\inv(\theta_k) \tau^{-(k+1)}(x^2) = \alpha^{k+1}\tau\inv(\theta_{k+1}) \tau^{-(k+1)}(x),
\end{align*}
so the result holds by induction. A similar inductive argument shows that
\[
x\theta_k = \alpha^{-k}\tau^{-2}(\theta_k) \tau^{-k}(x).
\]

By the above computations, we see that
\begin{align*}
x\Omega &= x(\theta_m + \tau\inv(\theta_m) + \tau^{-2}(\theta_m)) \\
&= \alpha^{-m} \tau^{-2}(\theta_m) \tau^{-m}(x) + \theta_m \tau^{-m}(x) + \alpha^{m} \tau\inv(\theta_m)\tau^{-m}(x) \\
&= \left(\tau^{-2}(\theta_m) + \theta_m + \tau\inv(\theta_m)\right)\tau^{-m}(x) \\
&= \Omega \tau^{-m}(x).
\end{align*}
Computations for $y\Omega$ and $z\Omega$ are similar, and so if $3 \nmid m$ then $\Omega$ is normal but not central, and it induces the automorphism $\tau^{-m}$.

We also have
\begin{align*}
x\Phi &= x(\theta_m + \omega\tau\inv(\theta_m) + \omega^2\tau^{-2}(\theta_m)) \\
&= \left(\tau^{-2}(\theta_m) + \omega\theta_m + \omega^2\tau\inv(\theta_m)\right)\tau^{-m}(x) \\
&= \omega\Phi \tau^{-m}(x).
\end{align*}
In particular, if $3 \mid m$, then $x \Phi = \Phi \omega x$. Computations for $y\Phi$ and $z\Phi$ are similar, and so if $3 \mid m$ then $\Phi$ is a normal element which induces the automorphism $\rho$. On the other hand, if $3 \nmid m$, then the computation shows that $\rho \circ \tau^{-m} \in \Oz(S_\alpha)$. Since we showed above that, in this case, $\tau^{-m} \in \Oz(S_\alpha)$, we see that $\grp{\tau, \rho} \leq \Oz(S_\alpha)$.

In particular, for any $\alpha$, we see that $\Oz(S_{\alpha})$ is nontrivial.
\end{proof}

Since we have shown that in all cases, $\rho$ is an ozone automorphism of $S_{\alpha}$, we obtain the following corollary.

\begin{corollary}
\label{cor.special1}
The center $Z(S_{\alpha})$ is contained in the $3$-Veronese of $S_{\alpha}$.
\end{corollary}

In the case that $3 \nmid m$, we can compute the entire ozone group. 

\begin{corollary}
\label{cor.special2}
Suppose $\alpha$ is a primitive $m$th root of unity and $3 \nmid m$. Then $\Oz(S_{\alpha}) = \grp{\tau, \rho} \cong \ZZ_3 \times \ZZ_3$.
\end{corollary}
\begin{proof}
By Lemma~\ref{lem.special1}, we have $\grp{\tau, \rho} \leq \Oz(S_\alpha)$. It remains to prove the reverse inclusion. Let $\phi \in \Oz(S_\alpha)$.
Then $\phi$ is necessarily a graded automorphism of $S_\alpha$ and,
moreover, $\phi(x^3)=x^3$. 
We observe that a $\kk$-basis for $(S_\alpha)_3$ is
\[ x^3, x^2y, x^2z, xyx, xzx, yx^2, yxy, xyz, yzx, zxy.\]

Write $\phi(x)=a_1x+a_2y+a_3z$.
Then the coefficient of $xyz$ in $\phi(x^3)$ is $a_1a_2a_3$. So, at least one of these must be zero. Suppose $a_3=0$. Then we have
\[ \phi(x^3) = (a_1^3 + a_2^3)x^3 + a_1^2a_2(x^2y+xyx+yx^2) + a_1a_2^2(xzx+yxy+x^2z).\]
Hence, one of $a_1,a_2$ must be zero. A similar argument holds if we had chosen $a_1$ or $a_2$ above.

Consequently, on the vector space $(S_{\alpha})_1 = \kk x + \kk y + \kk z$, $\phi$ acts as one of the following matrices:
\[
\begin{bmatrix}a & 0 & 0 \\ 0 & b & 0 \\ 0 & 0 & c \end{bmatrix}, \qquad  \begin{bmatrix} 0& 0 & c \\ a & 0 & 0 \\ 0 & b & 0\end{bmatrix}, \qquad \begin{bmatrix} 0 & b & 0 \\ 0 & 0 & c \\ a & 0 & 0 \end{bmatrix},
\]
where $a,b,c \in \kk$ are third roots of unity.

Since $\phi$ is an automorphism, then it respects the defining relations and so we have $ab=c^2$, $ac=b^2$, and $bc=a^2$. Thus, either $a=b=c$ or else $a,b,c$ are distinct third roots of unity.

Fix a primitive third root of unity $\omega$. For integers $i,j,k$ with $\{i,j,k\} = \{0,1,2\}$, let $\nu_{ijk}$ be the automorphism of $S_\alpha$ defined by $\nu_{ijk}(x)=\omega^i x$, $\nu_{ijk}(y)=\omega^j y$, and $\nu_{ijk}(z)=\omega^k z$. An easy check shows that the element $t=(z^2y)^m + (x^2z)^m + (y^2x)^m$ is central in $S_\alpha$. However, $\nu_{ijk}(z^2y)^m = \omega^{m(2k+j)}z^2$ and since $3 \nmid m$, $m(2k+j) \not\equiv 0 \mod 3$. Thus, $\nu_{ijk}$ does not fix $t$, and so $\nu_{ijk} \notin \Oz(S_\alpha)$. The result follows.
\end{proof}

We now consider the general case. 
For the remainder, we assume that $A=S(a,b,c)$ is PI and elliptic (that is, $E$ is a smooth elliptic curve).
Since $A$ is quadratic, we can consider the three 
relations in \eqref{E0.4.1} to be elements of $A_1\otimes A_1$ i.e. functions on $\PP(A_1) \times \PP(A_1)$.
Define the multilinearization correspondence 
$\Gamma\to \PP(A_1) \times \PP(A_1)$ as the zero locus of the 
quadratic relations of $A$ considered as functions on 
$\PP(A_1) \times \PP(A_1)$. By \cite{ATV}, we have 
$E=\pi_1(\Gamma)=\pi_2(\Gamma)$ is a cubic in $\PP(A_1)$ 
where $\pi_i$ is the $i$th projection. Furthermore, 
the correspondence $\Gamma$ is the graph of a translation
$\sigma$ of $E$.

The embedding of $E=\pi_1(\Gamma)$ in $\PP(A_1)$ furnishes 
$E$ with a very ample line bundle $\cL=\cO_{\PP(A_1)}\restrict{E}$. 
From the data $(E,\cL, \sigma)$, we form the twisted homogeneous 
coordinate ring $B=B(E,\cL, \sigma)$ (see \cite{AV}). One of the 
main results of \cite[Theorem 2]{ATV} is that $A/gA$ is 
isomorphic to $B$ where $g$ is the central element (c.f. 
\cite{AS} Equation (10.17))
\[ g = c(c^3 - b^3) y^3 + b(c^3 - a^3) yxz 
+ a(b^3 - c^3) xyz + c(a^3 - c^3) x^3.\]

Since we are interested in nondegenerate
Sklyanin algebras which are finite over their centers, the 
automorphism $\sigma$ is a translation of finite order, so we can assume
$\sigma = t_p$ where $p$ is an $n$-torsion 
element $p\in E_n$ (c.f. \cite[Theorem II]{ATV2}) with $n\ne 3$ (c.f. \cite[Theorem 3]{ATV} for why $n\ne 3$).

There is a group homomorphism
\begin{align}\label{E2.3.1}
\theta:\Oz(A)\to \Aut(E)
\end{align}
defined as follows. Given $\tau\in \Oz(A)$, it restricts to a 
linear map $A_1\to A_1$. Since $\tau$ fixes the center of $A$, 
in particular $\tau(g)=g$. Hence $\tau$ induces an automorphism 
of $E$, the zero locus of $g$, in $\PP(A_1)$. It is clear that 
$\mathrm{ker}(\theta)\cong\mu_j$ for some integer $j$ 
since the only automorphism of $\PP(A_1)$ which fix $E$ is the 
identity automorphism, which comes from scalar multiplication 
in $A_1$. Thus we have an exact sequence of groups
\[ 1 \to \mu_j \to \mathrm{Oz}(A) \to \Aut(E).\]

\begin{lemma}
\label{yylem2.4}
For any $\tau\in\Oz(A)$ the induced automorphism $\theta(\tau)$ 
of $E$ commutes with $\sigma$. Let $n$ denote the order of $\sigma$. If $n>2$, then $\theta(\tau)$ is a translation.
\end{lemma}

\begin{proof}
For ease of notation, we write $\tau\star x= \theta(\tau) x$ for any $x\in E$. Since $\tau$ is an automorphism of $A$, it induces an 
automorphism of $\Gamma\in\PP(A_1)\times\PP(A_1)$, so  
$(\tau\star x, \tau \star \sigma(x))\in \Gamma$. Since 
$\Gamma$ is the graph of $\sigma$, the fiber 
$\pi_1^{-1}(\tau\star x)=\left\{(\tau\star x, 
\sigma(\tau\star x))\right\}$ is a singleton. Hence 
$\tau\star \sigma(x)=\sigma(\tau\star x)$, so 
$\theta(\tau)$ and $\sigma$ commute. 

Next we show that $\theta(\tau)$ is a translation. By the 
basic theory of elliptic curves, we can write 
$\theta(\tau)=\varphi t_a$ where 
$\varphi\in\mathrm{Aut}_{\text{group}}(E)$ and $a\in E$, where $t_a$ denotes translation by $a$. Thus we have to show that $\varphi$ is the identity. We will use $+$ to indicate addition on the elliptic curve $E$ with respect to some fixed base point.

Recall that 
$\mathrm{Aut}_{\text{group}}(E_{\CC})=\mu_k$ where $k=2,4,6$. We check that for each possible value of $k$, that $n>2$ implies $\theta(\tau)$ is a translation. Most of the cases are handled with the following claim: if $\varphi$ has even order, then $p$ is $2$-torsion (which contradicts the assumption of $n>2$).

Proof of claim: Suppose the order of $\varphi$ is $2j$ for some integer $j$. Since $\theta(\tau)^j$ and $\sigma=t_p$ commute, and 
\begin{align}\label{translationcommutation}
     t_p \theta(\tau)^j &= \theta(\tau)^j t_{\varphi^{-j}(p)}
\end{align}
we have $\varphi^j(p)=p$. Now $\varphi^j$ is multiplication by $-1$, so $p$ is $2$-torsion. 

In other words, we have just shown that $\varphi$ must have odd order. In the cases where $k=2,4$, this forces $\varphi$ to be the identity, so $\theta(\tau)$ must be a translation. In the case of $k=6$, we assume $\varphi$ has order $3$. Now $\theta(\tau)$, $\theta(\tau)^2$ both commute with $\sigma$, so by \eqref{translationcommutation} we get $\varphi(p)=p$ and $\varphi^2(p)=p$. Now $\varphi$ is multiplication by a primitive cube root of unity, say $\xi$, so $\xi p=p$ and $\xi^2 p =p$. This gives 
\[ (1+\xi+\xi^2)p=3p.\]
The left hand side is zero since $\xi$ is a primitive cube root of unity, which shows that $p$ is $3$-torsion, and this is a contradiction.
Thus $\varphi$ is the identity, so $\theta(\tau)$ is a translation.
\end{proof}

We remind the reader of some terminology from classical 
algebraic geometry. Let $\cL$ be a line bundle on 
a projective variety $E$. Suppose $s\in H^0(E,\cL)$ is 
a section, we denote by $D:=(s)\in\Div(E)$ the divisor 
associated to $E$. We say that $D$ is a member of the (complete) 
linear system $\PP(A_1)=|\mathcal{L}|$ and write 
$D\in |\cL|$. We will depart from earlier notation and use $+$ to denote addition in $\Div(E)$. Given an automorphism $\sigma$ of $E$, 
denote by $D_\sigma^n$ the divisor 
$D+\sigma^{\ast}(D)+\cdots+(\sigma^{\ast})^{n-1}(D)$. 

\begin{lemma}\label{yylem2.5}
Let $(E,\cL,\sigma)$ be the triple associated to an elliptic 
Sklyanin algebra. For $a\in E$, let 
$\fd_a$ be the subvariety of $|\cL|$ consisting 
of $D$ such that the associated divisor $D_\sigma^n$ is 
$t_a^\ast$-invariant. If $t_a\ne \sigma^j$ for any integer $j$, 
then $\fd_a$ is $1$-dimensional.
\end{lemma}
\begin{proof}
Let $D=P+Q+R\in \fd_a$ where $P,Q,R\in E$ are distinct. Since $D_\sigma^n$ is $t_a^\ast$-invariant, either $t_a^\ast(P)$ is in the $\sigma^\ast$-orbit of $P$, or it is not. In the first case, we get $t_a=\sigma^j$ for some integer $j$, which contradicts our hypothesis. Hence 
$t_a^\ast(P)=(\sigma^\ast)^k(Q)$ for some integer $k$ and $Q$ 
not in the $\sigma^\ast$-orbit of $P$. Given any $D\in\fd_a$, we show that $D$ is fully determined by $P$. First, we can rearrange the previous equation to get $Q(P):=Q=(\sigma^\ast)^{-k}t_a^\ast(P)$. Second $R(P):=R$ is fully determined by $P$, as it is the negation of the sum of $P$ and $Q(P)$ under the elliptic curve addition law.  This gives a surjective map $E\to \fd_a$ by sending $P$ to $(P,Q(P),R(P))$, hence $\fd_a$ is $1$-dimensional.
\end{proof}

We are now ready to prove the following theorem.

\begin{theorem}\label{yythm2.6}
Let $n > 3$ and suppose that $A$ is a PI three-dimensional elliptic Sklyanin algebra whose defining automorphism has order $n$.
Let $\tau \in \Oz(A)$. Then $\theta(\tau) = 1$.

Furthermore, if $n$ is 
not divisible by $3$, then $\tau=1$. If $n$ is divisible by $3$,
then $\tau$ is scalar multiplication by a cube root of unity. 
\end{theorem}
\begin{proof}
First we show that $\theta(\tau)=1$. By Lemma~\ref{yylem2.4} 
we have $\theta(\tau)$ is a translation. Let $\theta(\tau)=t_a$ 
for some $a\in E$. 

Since any element of $\Oz(A)$ fixes $g$, it induces a 
$Z(A)/gZ(A)$-graded algebra automorphism of $B=A/gA$. This gives a 
group homomorphism
\[
\Oz(A)\to \Aut_{Z(A)/gZ(A)\text{-gralg}}(B).
\]
This map is injective since $A_1=B_1$, and a graded automorphism 
of $A$ is defined by a linear map $A_1\to A_1$, and the same 
linear map gives the graded automorphism of $B$. Moreover, the 
group homomorphism $\theta$ defined in \eqref{E2.3.1} factors 
through the map above. This shows that $\tau$ induces a $Z(A)/gZ(A)$
automorphism of $B$. 

By \cite[Proposition 1]{AST}, we can choose generic generators 
$x,y,z$ for $A$ such that the $n$th power of their images in $B$, $\bar{x}^n,\bar{y}^n,\bar{z}^n$, 
generate $Z(B)$ and they lift to generators of $Z(A)$. If 
$\tau$ fixes the center, then $\tau(\bar{x})^n=\bar{x}^n$. Let $D$ 
be the divisor on $E$ defined by 
$\bar{x}=0$. By the multiplication rule in $B$ 
(c.f. \cite[Equation (1.2)]{AV}), we see that $(x^n)=D_\sigma^n$. 
Then $\tau(\bar{x})^n=\bar{x}^n$ implies $D_\sigma^n$ is 
$\theta(\tau)^\ast=t_a^\ast$-invariant. By 
\cite[Theorem]{AST}, such divisors are dense in 
$\PP(A_1)=|\cL|$. By Lemma~\ref{yylem2.5}, we must have 
$t_a=\sigma^j$ for some integer $j$. By \cite[Equation (1.7)]{ATV} 
the map $\sigma$ quadratic, but $t_a$ is linear, so $j=0$ and 
$t_a=\theta(\tau)=1$. 

This implies $\tau$ is scalar multiplication by, say $\zeta$. To 
see which $\zeta$ are allowed, note that $g$ has degree $3$ and 
the other three generators of the center has degree $n$. For $\tau$ 
to fix these central elements, $\zeta$ must be simultaneously a 
cube root and an $n$th root of unity. If $3$ does not divide $n$, 
this means $\zeta=1$ and we are done.  If $3$ divides $n$, then $\zeta$
must be a cube root of unity. 
\end{proof}

\begin{theorem}[Theorem~\ref{yythmintro0.5}]
\label{yythm0.5}
Let $\kk$ be algebraically closed.
Let $A$ be a PI three-dimensional Sklyanin algebra whose defining automorphism has order $n \geq 2$. 
Suppose that $A$ is not a skew polynomial ring.
\begin{enumerate}
\item[(1)] 
If $n\geq2$ and $3 \nmid n$, then $\mathrm{Oz}(A)=\{1\}$.
\item[(2)] 
If $n>3$ and $3 \mid n$, then:
\begin{enumerate}
    \item if $A$ is elliptic, then $\Oz(A) \iso \cyc_3$, and
    \item if $A$ is not elliptic, then $\Oz(A) \neq \{1\}$.
\end{enumerate}
\item[(3)]
If $n=3$, then 
$A=S(1,0,\omega)$ where $\omega$ is a primitive sixth root of unity and $\Oz(A)=\ZZ_3$.
\end{enumerate}
\end{theorem}

\begin{proof}
We remark that every PI three-dimensional Sklyanin algebra is either elliptic or of type $\mathrm{P}$, $\mathrm{S}_1$, or $\mathrm{S}_3$ (see, e.g., \cite{IM1}). If $A$ is of type $\mathrm{P}$ or $\mathrm{S}_1$, then $A$ is skew polynomial.
The cases of $n=2$ is considered in Example~\ref{yyexa2.1}, while $n=3$ is considered in Lemmas~\ref{yylem2.2} and \ref{yylem2.3}.

If $A$ is of type $\mathrm{S}_3$, then it is one of the algebras $S_{\alpha}$ (see equation \eqref{S3}). The result follows from Lemma~\ref{lem.special1}. Note that if $3$ does not divide the order of $\alpha$, then $\Oz(S_{\alpha}) = \cyc_3 \times \cyc_3$ by Corollary~\ref{cor.special2}.

Finally, the elliptic case follows from the last statement of Theorem~\ref{yythm2.6} and 
observing that the center of $A$ is generated by elements with 
degrees $3$ and $n$, both of which are divisible by $3$, so 
scalar multiplication by any cube root of unity
is a graded $\kk$-algebra automorphism of $A$ which fixes the center.
\end{proof}

\begin{theorem}[Theorem~\ref{yythmintro0.6}]
\label{yythm0.6}
Let $A$ be a PI quadratic AS regular algebra of global 
dimension 3 with trivial ozone group. Then $A$ is isomorphic 
to one of the following:
\begin{itemize}
\item 
an elliptic Sklyanin algebra whose defining automorphism has order $n \neq 1$ such that $3 \nmid n$, or
\item 
one of the algebras $B_q$ {\rm{(}}see \eqref{E1.4.1}{\rm{)}} 
where $q \neq 1$ is a root of unity and $3$ does not divide the order of $q$.
\end{itemize}
\end{theorem}

\begin{proof}
We refer to the classification in Table 1 of \cite{IM1}. 

If $A$ is of type $\mathrm{P}$ or $\mathrm{S}_1$, then $A$ is a 
skew polynomial ring. Hence, $A$ has nontrivial ozone group.

If $A$ is of type $\mathrm{EC}$, then $A$ is 
an elliptic Sklyanin algebra. See Theorem~\ref{yythm2.6}.

If $A$ is of type $\mathrm{S}_3$, then $A$ is a non-elliptic Sklyanin algebra, and has a nontrivial ozone group by Lemma~\ref{lem.special1}.

If $A$ is of type $\mathrm{S}'$ or $\mathrm{TL}$, then $A$ is 
a quantum Heisenberg algebra. See Proposition~\ref{yypro1.4}.

If $A$ is of type $\mathrm{NC}$, then $A$ is isomorphic to 
one of the $B_q$. See Proposition~\ref{yypro1.7}.

If $A$ is of type $\mathrm{T}_1$, $\mathrm{T}_3$, 
$\mathrm{T}'$, $\mathrm{CC}$, or $\mathrm{WL}$, then $A$ is 
not finite over its center.
\end{proof}

\section{Proof of Theorem~\ref{yythmintro0.7}}
\label{yysec3}

The main goal of this section is to prove Theorem~\ref{yythmintro0.7} from the introduction (proved as Theorem~\ref{yythm0.7} below).
That is, we wish to show that for a noetherian PI AS regular 
algebra $A$, the order of $\Oz(A)$ is finite and divides 
$\rk_Z(A)$.

Throughout this section, we assume that $A$ is a noetherian 
PI AS regular algebra. Consequently, $A$ is a domain. We let 
$\overline{A}=A\#\kOz(A)$.

\begin{definition}
\label{yydef3.1} 
Let $A$ be a noetherian PI AS regular algebra. The 
\emph{ozone-extended center} of $A$ is defined to be 
$Z(\overline{A})$.
\end{definition}

We often use the notation $\overline{Z}$ for $Z(\overline{A})$ 
in case the algebra $A$ is understood. Identifying $Z=Z(A)$ 
with $Z\# 1$, we see that $Z$ is a subring of $Z(\overline{A})$.

By Lemma~\ref{yylem3.4}(3,4) below, $\overline{A}$ is an 
Auslander regular Cohen--Macaulay prime PI ring. By Lemma 
\ref{yylem3.4}(4) below, $\overline{Z}$ is a connected graded 
commutative domain.

In the first half of this section, we explore the connection 
between the center of $A$, the ozone-extended center of $A$, and 
the center of the invariant ring $A^{\Oz(A)}$. In doing so, we 
are able to show that $\Oz(A)$ is finite and divides $\rk_Z(A)$ in the second half of the section (see Theorem~\ref{yythmintro0.7}).

\begin{lemma}\label{yylem3.2}
Let $G$ be an abelian subgroup of $\Aut(A)$. Then, as a 
$\kk$-vector space, the center $Z(A\# \kk G)$ is spanned 
by $\{ a \# \eta_a \mid a \in A^{G} 
\text{and $a$ is normal in $A$}\}$.
\end{lemma}
\begin{proof}
Let $f$ be a nonzero element in the center $Z(A\# \kk G)$ 
and write it as a finite sum $\sum_{i=1}^{n} a_i \# g_i$, where 
$0\neq a_i \in A$ and $g_1,\cdots, g_n$ are distinct elements 
in $G$. For any $h \in G$, we have
\[\sum_{i=1}^{n}a_i \# g_i h = \left(\sum_{i=1}^n a_i \# g_i\right)(1 \# h) 
= (1 \# h)\left(\sum_{i=1}^{n} a_i \# g_i\right) = \sum_{i=1}^n h(a_i) \# g_ih
\]
and so each $a_i \in A^{G}$. Further, for any $b \in A$, we have
\[
\sum_{i=1}^{n} a_i g_i(b) \# g_i = \
\left(\sum_{i=1}^n a_i \# g_i\right)(b \# 1) = (b \# 1)
\left(\sum_{i=1}^n a_i \# g_i\right) = b\sum_{i=1}^n a_i \# g_i.
\]
Since this is true for all $b \in A$, we see that each 
$a_i$ is a normal element and $g_i = \eta_{a_i}$.

Conversely, if $0\neq a \in A$ is normal in $A$, then for 
any $b \# g \in A\# \kk G$, we have
\[
(a \# \eta_a)(b \# g) = a \eta_a(b) \# \eta_a g = 
ba \# g \eta_a = (b \# g)(a \# \eta_a)
\]
so $a \# \eta_a \in Z(A\# \kk G)$.
\end{proof}

The following lemma generalizes \cite[Corollary 0.6]{RRZ}. 
In particular, it shows that if $\Oz(A)$ is abelian and 
has trivial homological determinant, then $A \# \kk \Oz(A)$ 
is CY.

\begin{lemma}
\label{yylem3.3}
Let $A$ be a PI AS regular algebra. Let $G$ be an abelian 
subgroup of $\Autgr (A)$ which contains the Nakayama 
automorphism $\mu_A$. Then $A \# \kk G$ is CY if and only 
if $G$ has trivial homological determinant.
\end{lemma}

\begin{proof}
By \cite[Theorem 0.2]{RRZ},
\[
\mu_{A \# \kk G} 
= \mu_A \# (\mu_{\kk G} \circ \Xi^{\ell}_{\hdet})
\]
where $\Xi^{\ell}_{\hdet}$ is the left winding automorphism 
of $\hdet$, that is, $\Xi^{\ell}_{\hdet}(g) = \hdet(g)g$ 
for all $g \in G$.

Suppose first that $A \# \kk G$ is CY so $\mu_{A \# \kk G}$ 
is inner. Then there exists an invertible element 
$r \in A \# \kk G$ such that for all $a \in A$ and $g \in G$,
\[
a \# g = r\inv \mu_{A \# \kk G}(a \# g) r 
= r\inv(\mu_A(a) \# \hdet(g) g)r.
\]
In particular, $r(1\# g) = (1 \# \hdet(g) g)r$. Using the grading on $A \# \kk G$ (where we place $\kk G$ in degree $0$), write $r = r_0 + \cdots + r_n$ where each $r_i$ has degree $i$. 
Then we have $r_0(1 \# g) = (1 \# \hdet(g) g)r_0$. 
Since $r$ is invertible, $r_0 \neq 0$, so
$r_0 = \sum \alpha_i \# g_i$ for some $\alpha_i \in \kk^\times$ and distinct $g_i \in G$. Then $\alpha_1(1 \# g_1 g) = \alpha_1(1 \# \hdet(g) g g_1)$ and since $G$ is abelian, $\hdet(g) = 1$ for all $g \in G$.

Conversely, suppose that $G$ has trivial homological 
determinant. By assumption, $\mu_A \in G$. Then for all 
$a \in A$ and $g \in G$, we have
\[
\mu_{A \# \kk G}(a \# g) = \mu_A(a) \# g 
= (1 \# \mu_A) (a \# g) (1 \# \mu_A)\inv
\]
and so $A \# \kk G$ is CY.
\end{proof}

\begin{lemma}
\label{yylem3.4}
Let $A$ be a noetherian PI AS regular algebra and let $G$ be a 
finite subgroup of $\Aut_{\gr}(A)$. Let $B = A\# \kk G$. Then the following hold:
\begin{enumerate}
\item[(1)]
$B$ is generalized AS regular in the sense of
\cite[Definition 3.3]{RRZ}.
\item[(2)]
$B$ is graded injectively smooth in the sense of 
\cite[p.990]{SZ} and graded smooth in the sense of 
\cite[p.1009]{SZ}.
\item[(3)]
$B$ is Auslander regular, Macaulay, and Cohen--Macaulay.
\item[(4)]
$B$ is prime and the center of $B$ is a connected graded domain.
\item[(5)] The rank of $A$ over $A^G$ is equal to $|G|$.
\end{enumerate}
\end{lemma}

\begin{proof}
(1) This is \cite[Theorem 4.1]{RRZ}.

(2) Let $M$ be a simple graded $B$-module. Then it is finite
dimensional over $\kk$. By \cite[Theorem 4.1]{RRZ}, the rigid 
dualizing complex over $B$, denoted by $R_B$, is isomorphic to
${^1 B^{\mu_B}}(-\fl)[-d]$ where $d$ is the global dimension
of $A$ (and the global dimension of $B$). By locally duality
\cite[Remark 3.6]{RRZ}, 
\[ \Ext^d_B(M,B(-\fl))
=R\Hom_B(M,R_B)\cong R\Gamma_{\fm_B}(M)^\ast
=M^\ast\neq 0.\]
Therefore $B$ is graded injectively smooth in the sense of 
\cite[p.990]{SZ}. Since $B$ has global dimension $d$, it is 
graded smooth in the sense of \cite[p.1009]{SZ}.

(3) This follows from part (2) and \cite[Theorem 3.10]{SZ}.
Further, Macaulay in the sense of \cite{SZ} is equivalent
to Cohen--Macaulay since the Krull dimension of a finitely
generated $B$-module is equal to its GK dimension. 

(4) First we claim that  $Z(B)$, the center of $B$, is 
connected graded (where the grading on $B$ is inherited from 
the grading on $A$, placing elements of $\kk G$ in degree $0$). 
If $f\in Z(B)_0$, we write it as $f=\sum_{g\in G} c_g\# g$ 
where $c_g\in \kk$. Since $f\in Z(B)$, for every $a\in A$, we 
have
\[ 0= (a\# 1) f-f (a\# 1)=\sum_{g\in G} c_g(a-g(a))\# g. \]
This implies that $c_g (a-g(a))=0$ for all $g \in G$. But 
for all nonidentity $g \in G$, there exists $a \in A$ such 
that $a \neq g(a)$. Hence, $c_g = 0$ for all $g \neq 1$. 
Hence, $f \in \kk \# 1$ so the assertion follows.

By \cite[Theorem 5.4]{SZ} and part (3), $B$ is a direct
sum of prime algebras. If $B$ is not prime, then 
$B=\bigoplus_{i=1}^n C_{i}$ with $n\geq 2$ and each $C_i$ 
is an ${\mathbb N}$-graded prime algebra. Thus the degree 
zero part of the center $Z(B)$ has dimension at least 2, 
yielding a contradiction. Therefore $B$ is prime. 
Consequently, the center of $B$ is a domain.

(5) This follows from part (4) and \cite[Proposition 1.5]{KWZ1}.
\end{proof}

For any finite subgroup $G$ of $\Aut_{\gr}(A)$, one can define 
the \emph{pertinency} of the $G$-action on $A$ to be
\[ \p(A,G):=\GKdim A-\GKdim 
\left((A\#\kk G)/(\sum_{g\in G} 1\# g)\right),\]
see \cite[Definition 0.1]{BHZ1}. 

In general, we do not know when $Z'$ and $\overline{Z}$ are 
isomorphic as graded algebras. However, we can provide an 
additional criteria which allows us to relate the three 
centers.

By \cite[Theorem 0.3]{BHZ1}, if $\p(A,G)\geq 2$, then the 
natural morphism (called the \emph{Auslander map})
$A\# \kk G\to \End_{A^{G}}(A)$ is an isomorphism. 
The following result is a consequence of this fact.

\begin{proposition}
\label{yypro3.5}
Let $A$ be a noetherian PI AS regular algebra with center 
$Z$. If $\p(A, \Oz(A))\geq 2$, then there is a natural 
injective morphism from $Z'\to \overline{Z}$. In this case, 
by abuse of notation, we can relate three centers by
\[ Z\subseteq Z'\subseteq \overline{Z}.\]
\end{proposition}

\begin{proof} 
Recall that $A'=A^{\Oz(A)}$. By the hypothesis 
$\p(A, \Oz(A))\geq 2$ and \cite[Theorem 0.3]{BHZ1},
there is a natural isomorphism of graded algebras
\[ f: \End_{A'} (A)\to A\# \kOz(A)=:\overline{A}.\]
Since $Z'$ is the center of $A'$ and $A$ is considered as 
a right $A'$-module, $Z'$ is in the center of $\End_{A'}(A)$.
Since $f$ induces an isomorphism between the centers of 
$\End_{A'}(A)$ and $\overline{A}$, $f$ restricts to an 
injective morphism from $Z'$ to $\overline{Z}$.
\end{proof}

We do not know if the conclusion of Proposition~\ref{yypro3.5} 
holds without the hypothesis of $\p(A, \Oz(A))\geq 2$.

\begin{example}
\label{yyexa3.6}
Let $A$ be the quantum Heisenberg algebra from \eqref{E1.2.2} 
in the case $q=-1$. By Lemma~\ref{yylem1.3}, $Z=Z(A)$ is 
generated by $t_1:=x^2$, $t_2:=y^2$, $t_3:=z^2$, and 
$t_4:=(xy-yx)z$. It is easy to then show that
\[ Z \iso \kk[t_1,t_2,t_3,t_4]/(t_4^2 + 4t_1t_2t_3 - t_3^3).\]
We compute $\overline{Z}$ and $Z'$.

By Proposition~\ref{yypro1.4}, $\Oz(A) = \{\Id, \phi\}$ where 
$\phi(x) = -x$, $\phi(y) = -y$, and $\phi(z) = z$. Let $h$ be 
a homogeneous element of degree $d$ in $\overline{Z}$. Write 
$h$ as $h_0 \# \Id + h_1\# \phi$ where $h_0,h_1\in A_d$. Then 
$h_0\in Z$, $h_1\in A^{\langle \phi\rangle}$, and 
$f h_1=h_1 \phi(f)$ for all $f\in A$. A routine calculation 
shows that $\overline{Z}$ is generated by 
$s_1:=x^2\# \Id, s_2:=y^2\# \Id, s_3:=
z\# \phi$, and $s_4:=(yx-xy)\# \phi$. So 
\[\overline{Z}\cong \kk[s_1,s_2,s_3,s_4] /(s_3^4-4s_1s_2-s_4^2)\]
which has hypersurface isolated singularities. The rank of 
$\overline{A}$ over $\overline{Z}$ is 4. It is clear that $Z$ 
and $\overline{Z}$ are not isomorphic.

By definition, $A'=A^{\Oz(A)}$. It is easy to see that $A'$ is 
a commutative domain generated by $u_1:=x^2, u_2:=y^2, u_3:=z$, 
and $u_4:=yx-xy$ and subject to the relation 
$u_3^4-4u_1u_2-u_4^2=0$. So $Z'=A'\cong \overline{Z}$ and it 
has an isolated singularity at the origin. 

The pertinency of the $\Oz(A)$ action on $A$ is 
$\p(A, \Oz(A)) = 3$ as both $x\# 1$ and  $y\# \phi$ are in the 
ideal generated by the integral. The homological determinant 
of $\phi$ is 1.
\end{example}

Let $B$ be a prime PI algebra with center $Z(B)$. 
Then the \emph{PI-degree} of $B$ is defined to be 
\begin{align}\label{E3.6.1}
\PIdeg(B):=\sqrt{\rk_{Z}(B)}
\end{align}
which is a positive integer (c.f. \cite[Chapter 13]{MR}).
Consequently, $\rk_{Z}(B)$ is a perfect square.

\begin{remark}
\label{yyrem3.7}
If $\left|\Oz(A)\right|$ is not a perfect square, then 
$\overline{Z}\neq Z$. To see this we note that 
$\rk_Z(\overline{A})= \rk_{Z}(A)\left|\Oz(A)\right|$. 
If $\left|\Oz(A)\right|$ is not a perfect square but 
$\rk_Z(A)$ is, then $\rk_Z(\overline{A})$ is not a 
perfect square. So $\overline{Z}$ strictly contains $Z$.

On the other hand, there are examples of $A$ such that 
$\p(A, \Oz(A))=1$ and $Z'=\overline{Z}=Z$.
See, e.g., the skew polynomial ring in Example \ref{yyexa1.2}.
\end{remark}

We will use the following easy and well-known lemma.

\begin{lemma}
\label{yylem3.8}
Let $B$ be a prime PI algebra with center $Z$. Let $C$ 
be a prime subalgebra of $B$ containing $Z$. Then
\begin{enumerate}
\item[(1)] 
$\rk_{Z}(C)$ divides $\rk_Z(B)$.
\item[(2)] 
The PI-degree of $C$ divides the PI-degree of $B$.
\item[(3)] 
$\rk_{Z}(Z(C))$ divides $\rk_{Z}(C)$ and $\rk_Z(B)$.
\end{enumerate}
\end{lemma}

\begin{proof} (1) After localization, we may assume that 
$Z=Z(B)$ is a field. So both $B$ and $C$ are semisimple
artinian. Write $C = M_{n\times n}(D)$ for some division 
algebra $D$. Then, by using the matrix unit of $C$, we obtain
$B = M_{n\times n}(E)$ with a prime ring $E\supseteq Z$ and 
$D\subseteq E$. So we have 
$\rk_Z(E)=\rk_D(E)\rk_Z(D)$. Thus 
\[\rk_Z(C)=n^2 \rk_Z(D)\mid n^2\rk_Z(E)=\rk_Z(B).\]

(2, 3) Let $Z'$ be the center of $C$. 
Then $Z\subseteq Z'\subseteq C$. So 
\[\rk_{Z}(C)=\rk_{Z}(Z')\rk_{Z'}(C).\] 
Thus both $\rk_Z(Z')$ and 
$\rk_{Z'}(C)$ divides $\rk_{Z}(C)$. Combining with part (1), 
$\rk_{Z'}(C)$ divides $\rk_Z(B)$. The assertions follow 
from the definition.
\end{proof}

\begin{proposition}
\label{yypro3.9} 
Let $A$ be a noetherian PI AS regular algebra with center 
$Z$ and let $A'$ be $A^{\Oz(A)}$. Then
\begin{enumerate}
\item[(1)]
$\left|\Oz(A)\right|=\rk_{A'}(A)$.
\item[(2)]
The order of $\Oz(A)$ is equal to $\rk_Z(A)$ 
if and only if $Z=A'$.
\item[(3)]
If $Z=A^{\Oz(A)}$, then the order of $\Oz(A)$ is a 
perfect square.
\end{enumerate}
\end{proposition}

\begin{proof}
(1) This follows from Lemma \ref{yylem3.4}(5).

(2) Suppose $\left|\Oz(A)\right|=\rk_Z(A)$. By part (1), 
$\rk_{A'}(A)=\rk_Z(A)$. So $\rk_Z(A')=1$. After 
localization $Q:=A'\otimes_Z F$ is 1-dimensional over 
$F:=Z_{(Z\setminus\{0\})^{-1}}$. Hence $Q=F$. By 
\cite[Corollary 1.2]{StaZ}, $Z$ is integrally closed. 
Hence $Z=A'$. The other direction follows from part (1).

(3) This follows from part (2) as $\rk_Z(A)$ is a perfect 
square.
\end{proof}

\begin{proposition}\label{yypro3.10} 
Let $A$ be a noetherian PI AS regular algebra with center 
$Z$. Let $G$ be any finite subgroup of $\Oz(A)$. Then the 
order of $G$ divides $\rk_{Z}(A)$.
\end{proposition}

\begin{proof} 
Let $B=A\# \kk G$. By Lemma \ref{yylem3.4}(4), $B$ is a prime 
algebra containing $Z$. Let $C$ be the localization $B_{S^{-1}}$ 
where $S:=Z\setminus\{0\}$ and let $F$ be the field of fractions 
of $Z$. Let $D=\End_{Z}(A)$. There is a natural map $f: B\to D$. 
Since $Z$ is in both $B$ and $D$, the image of $f$ has GK 
dimension equal to the GK dimension of $Z$ (which is equal to 
the GK dimension of $B$). Since $B$ is prime, $f$ is injective. 
Therefore we can consider $B$ as a subalgebra of $D$ via $f$. 
Let $E$ be the localization $D_{S^{-1}}$ which is an $F$-algebra. 
By an easy localization argument, we have 
\[E=\End_{F}(A_F)\cong M_{n\times n}(F)\]
where $n=\rk_Z(A)$. So $\rk_{F}(E)=n^2=(\rk_Z(A))^2$. Since $B$ 
is a prime subalgebra of $D$, $C$ is a prime subalgebra of $E$ 
and $C$ contains the center of $E$, namely, $F$. By Lemma 
\ref{yylem3.8}(1), $\rk_F(C)$ divides $\rk_F(E)$. Now we obtain 
\[|G|\rk_Z(A)=\rk_Z(B)=\rk_F(C) \;\mid \;\rk_F(E)=(\rk_Z(A))^2.\]
Hence $|G|$ divides $\rk_Z(A)$.
\end{proof}

The following lemma completes the proof of Theorem~\ref{yythmintro0.7}.

\begin{lemma}
\label{yylem3.11}
Let $A$ be a noetherian PI AS regular algebra with center $Z$.
Let $\Phi$ be an element in $\Oz(A)$.
\begin{enumerate}
\item[(1)] 
All eigenvalues of $\Phi$ are roots of unity. 
\item[(2)] 
There are no nonzero elements $x$ and $y$ in $A$ such that 
$\Phi(y)=y$ and $\Phi(x)=x+y$.
\item[(3)] 
Every element in $\Oz(A)$ has finite order. As a consequence, 
for any $\phi\in \Oz(A)$, there are no nonzero elements 
$x,y\in A$ such that $\phi(y)=cy$ and $\phi(x)= cx+y$ for some 
$c\in \kk^{\times}$.
\item[(4)] 
$\Oz(A)$ is an algebraic group.
\end{enumerate}
\end{lemma}

\begin{proof}
(1) Suppose to the contrary that $\Phi$ has an eigenvalue 
$\lambda$ that is not a root of unity. Let $x\in A_d$ be a 
corresponding eigenvector. Since $\rk_Z(A)$ is finite, there 
is an $n>0$ such that $\{1, x, x^2, \cdots,x^n\}$ are linearly 
dependent over $Z$. So we have $f_0,f_1,\cdots,f_n\in Z$ 
(not all zero) such that
\[f:=f_0+xf_1+x^2f_2+\cdots+ x^n f_n\]
is zero. We may assume that $n$ is minimal. 
Since $A$ is a domain, $f_0f_n\neq 0$ and $n>0$. Now 
\[ 0=\Phi(f)-\lambda^n f=
(1-\lambda^n) f_0+x(\lambda-\lambda^n)f_1+\cdots+x^{n-1}
(\lambda^{n-1}-\lambda^n)f_{n-1}\]
yields a contradiction to the minimality of $n$. Therefore the 
claim follows.

(2) Suppose to the contrary that such elements $x$ and $y$ 
exist. Now we fix a large integer $n$ and consider the 
monomials in $x$ and $y$. Let $e$ be a nontrivial linear 
combination of monomials in $x$ and $y$ of total degree $n$ 
with positive integral coefficients. For simplicity, we call 
such $e$ an \emph{effective polynomial}. We define $\deg_x(e)$ 
be the largest number of the $x$ component in one of the 
monomials appeared in $e$. If $e$ is an effective polynomial 
of $\deg_x 0$, then $e=c y^n$ with $c$ a positive integer. If 
$e$ is an effective polynomial of $\deg_x$ $n$, then 
$e=c x^n+$ other terms with $c$ a positive integer.

If $e:=e(x,y)$ is an effective polynomial of degree $\deg_x=i>0$, 
then $\Phi(e)$ is $e(x+y,y)$. After expanding the polynomial 
$e(x+y,y)$ (for example, in the free algebra generated in $x$ 
and $y$), one sees that 
\[e(x+y,y)=e(x,y)+g(x,y)\]
where $g(x,y)$ is an effective polynomial of degree $\deg_x=i-1$. 
In other words, $\Phi(e)-e$ is an effective polynomial of 
degree $\deg_x=i-1$. If we define 
$e(0):=x^n, e(1), \cdots, e(n)$, inductively, by 
\[e(i+1)=\Phi(e(i))-e(i)\]
for all $=0, 1,\cdots, n-1$, we obtain that $e(i)$ is an 
effective polynomial of degree $\deg_x=n-i$. For example, 
$e(n)=c y^n$ for some positive integer $c$.

Suppose now $n>\rk_Z(A)$. Then there are elements 
$f_0,f_1,\cdots, f_n \in Z$, not all zero, such that
\[L:=e(0) f_0+\cdots+ e(n) f_n\]
is zero. Then 
\[0=\Phi(L)-L=e(1) f_0+\cdots+ e(n) f_{n-1}.\]
By induction, we can show that $e(n)f_0=0$. Since $A$ is a 
domain, $f_0=0$. Now removing the term $e(0)f_0$ and doing 
a similar argument show that $f_1=0$. By induction, all $f_i=0$,
yielding a contradiction. The assertion follows.

(3) Let $\Phi$ be an element in $\Oz(A)$. Since $A$ is 
finitely generated, we can assume that $A$ is generated by 
$V:=\oplus_{d=0}^{N} A_d$ for some $N>0$. 
We only need to show that $\Phi$ has finite order when 
restricted to $V$.

By (1), every eigenvalue of $\Phi$ is a root of unity 
(and the order is bounded by $\rk_Z(A)$). 
After replacing $\Phi$ by a finite power of $\Phi$, we can 
assume that all eigenvalues of $\Phi$ is 1. By (2), $\Phi$
must be the identity. The main assertion follows.

Now suppose to the contrary that there are no nonzero
elements $x,y\in A$ such that $\phi(y)=c y$ and $\phi(x)= cx+y$
for some $c\in \kk^{\times}$. Since $\phi$ has finite order,
$c$ is a root of unity (say, of order $n$). Then
$\phi^n(nc^{n-1}y)=nc^{n-1}y$ and $\phi^n(x)=x+nc^{n-1}y$. 
This contradicts to (2).

(4) By \cite[Lemma 7.6]{RRZ}, $\Aut_{\gr}(A)$ is an affine 
algebraic group. It is clear that $\Oz(A)$ is a closed 
subgroup of $\Aut_{\gr}(A)$. So the assertion follows.
\end{proof}

Now we are ready to prove Theorem~\ref{yythmintro0.7}.

\begin{theorem}[Theorem~\ref{yythmintro0.7}]\label{yythm0.7}
Let $A$ be a noetherian PI AS regular algebra with center $Z$.
Then the order of $\Oz(A)$ is finite and divides $\rk_Z(A)$.
\end{theorem}

\begin{proof}
By Lemma \ref{yylem3.11}(3,4), $\Oz(A)$ is an algebraic group 
in which every element has finite order. By algebraic group 
theory, $\Oz(A)$ is finite.
By Proposition \ref{yypro3.10}, 
$\left|\Oz(A)\right|$ and divides $\rk_Z(A)$.
\end{proof}

As a consequence of Theorem~\ref{yythm0.7}, if $f$ is a 
nonzero normal element in $A$, then $f^n$ is central 
for some $n\leq \Oz(A)$. 

\section{A characterization theorem}\label{yysec4}

In this section we provide some basic facts about the ozone 
group of skew polynomial ring $A=\SP$, see \eqref{E0.7.1}
for the definition, as well as
some related algebras. We will assume throughout that $A$ is 
PI, which is equivalent to assuming that each $p_{ij}$ is a 
root of unity. One of the main results is Theorem~\ref{yythmintro0.8} (proved as Theorem~\ref{yythm0.8} below).

\begin{example}\label{yyex4.1} 
Let $a,b \in \ZZ^n$ and write $a=(a_1,\cdots,a_n)$ and 
$b=(b_1,\cdots,b_n)$. We define an order on $\ZZ^n$ by 
declaring $a<b$ if there is an $1\leq i\leq n$ such that 
$a_i<b_i$ and $a_j=b_j$ for all $j>i$. It is easy to check 
that this order is preserved by multiplication. In this way, 
$\ZZ^n$ is an ordered group. We set
\[ e_i=(0,\cdots, 0, 1,0,\cdots,0) \]
where $1$ is the $i$th position. Then 
\[ e_1<e_2<\cdots <e_n.\]
\end{example}

Let $G=\ZZ^n$ and consider $\SP$ as a $G$-graded algebra where 
$\deg(x_i)=e_i$. Then every $G$-homogeneous element is a monomial 
$x_1^{i_1}\cdots x_n^{i_n}$ up to a nonzero scalar. Let $\eta_i$ 
denote the automorphism of $\SP$ given by conjugation by $x_i$, 
that is
\[\eta_i(f) = x_i^{-1} f x_i, \quad \text{for all $f\in \SP$}\]
and let $O$ be the subgroup of $\Aut_{\gr}(\SP)$ generated by 
$\eta_1, \cdots,\eta_n$. Part (1) of the next lemma says that 
$O=\Oz(\SP)$.

\begin{lemma}
\label{yylem4.2} 
Let $A=\SP$ be PI and let $Z$ denote the center of $A$.  Then
\begin{enumerate}
\item[(1)] 
$\Oz(A)=O$,
\item[(2)] 
$Z=A^{\Oz(A)}$, and
\item[(3)]
$\Oz(A)$ is abelian and $\left|\Oz(A)\right|$ equals to $\rk_Z(A)$.
\end{enumerate}
\end{lemma}

\begin{proof}
(1) It is clear that $O\subseteq \Oz(A)$. It suffices to show 
that $\Oz(A)\subseteq O$. Consider $A$ as a $\ZZ^n$-graded 
algebra and note that $A$ is a domain. By Lemma 
\ref{yylem1.9}(2), every $\phi\in \Oz(A)$ is of the form 
$\eta_a$ where $a$ is a $\ZZ^n$-homogeneous. So, 
$a=x_1^{i_1}\cdots x_n^{i_n}$ up to a nonzero scalar and 
$\eta_a=\prod_{s=1}^n (\eta_{x_s})^{i_s}=
\prod_{s=1}^n (\eta_s)^{i_s}\in O$. The assertion follows.

(2) Note that $f\in Z$ if and only if $x_i f=fx_i$ for all
$i$ if and only if $\eta_i(f)=f$ for all $i$. The 
assertion follows by part (1).

(3) It is clear that $O$ is an abelian group. By part (1),
$\Oz(A)$ is abelian. It follows from Lemma \ref{yylem3.4}(5)
and part (2) that $\left|\Oz(A)\right|=\rk_{A^{\Oz(A)}}(A)=
\rk_Z(A)$.
\end{proof}

Next we will show that Question \ref{yyque0.10} has a positive
answer for a class of AS regular algebras. We start with 
the definition of a $G$-filtered algebra where 
$G$ is an ordered abelian group. We use the material in 
\cite[Section 6]{Zh3}. Let $A$ be an algebra with a filtration 
$\{F_g \mid g\in G\}$ of subspaces of $A$. For every $g\in G$, 
we define $F_{<g}$ to be $\sum_{h<g} F_h$. Suppose the 
filtration satisfies the following conditions:
\begin{enumerate}
\item[(F1)]
$F_g\subseteq F_h$ for all $g<h$ in $G$;
\item[(F2)]
$F_g F_h\subseteq F_{gh}$ for all $g,h\in G$;
\item[(F3)]
$A=\bigcup_{g\in G} (F_g-F_{<g})$;
\item[(F4)]
$1\in F_{u}-F_{<u}$ where $u$ is the identity element of $G$.
\end{enumerate}
Then we can define the corresponding associated graded algebra 
\[ \gr(A):=\bigoplus_{g\in G} F_g/F_{<g}\] 
with the $\kk$-linear multiplication determined by
$(a+F_{<g})(b+F_{<h}):=ab+F_{<gh}$. 
If $A$ has such a filtration, we call it a 
\emph{$G$-filtered algebra}. For our purpose we also assume 
\begin{enumerate}
\item[(F5)]
$\gr(A)$ is a $G$-graded domain.
\end{enumerate}
We define a map $\nu: A\to \gr(A)$ by $\nu(0)=0$ and 
$\nu(a):=a+F_{<g}$ for all $a\in F_g-F_{<g}$. This map is 
called the \emph{leading-term map}, and it is easy to see that 
\begin{enumerate}
\item[(L1)]
$\nu(c)=c$ for all $c\in \kk$;
\item[(L2)]
$\nu(a)\neq 0$ for all $a\neq 0$. 
\end{enumerate}
By (F5) and the definition of $\gr(A)$, we see that
\begin{enumerate}
\item[(L3)]
$\nu(ab)=\nu(a)\nu(b)$ for all $a,b\in A$.
\end{enumerate}
Some examples of $G$-filtered algebras can be found in 
\cite[Section 6]{Zh3}.

For the rest of this section we consider a version of 
filtered skew polynomial rings. 

\begin{definition}
\label{yydef4.3}
Let $A$ be a connected graded AS regular algebra with a fixed 
minimal algebra generating set of homogeneous elements 
$\bx:=\{x_1,\cdots,x_g\}$. We say $A$ is a \emph{filtered skew 
polynomial ring} if there is a set of homogeneous elements 
$\bt:=\{t_1,\cdots,t_n\}$ such that $\bx\subseteq\bt$ and $A$ 
is an iterated Ore extension $\kk[t_1][t_2; \sigma_2, \delta_2]
\cdots [t_n; \sigma_n,\delta_n]$ such that 
$\sigma_j(t_i)=p_{ij} t_i$ for all $1\leq i<j\leq n$ where 
$p_{ij}\in \kk^{\times}$. We call $\kk[t_1][t_2; \sigma_2, 
\delta_2] \cdots [t_n; \sigma_n,\delta_n]$, a 
\emph{filtered skew polynomial ring realization} of $A$.
\end{definition}

For $\alpha,\beta \in \kk$ with $\beta \neq 0$, the 
\emph{graded down-up algebra} is defined as
\begin{align}\label{E4.3.1}
A(\alpha,\beta) := \kk\langle x,y \rangle/
(x^2 y -\alpha xyx -\beta yx^2, xy^2-\alpha yxy-\beta y^2x).
\end{align}
It is well-known that $A(\alpha, \beta)$ is AS regular of global 
dimension three.

\begin{example}\label{yyexa4.4}
Consider the down-up algebra $A=A(2,-1)$. Setting $z=xy-yx$ we 
see that $zx=xz$ and $zy=yz$. Hence, $\{x,y,z\}$ is a generating 
set which may be presented as an iterated Ore extension of the 
above form.
\end{example}

Note that in a filtered skew polynomial ring, for $i < j$, we have
\[ t_j t_i=p_{ij} t_i t_j+\delta_j(t_i)\]
where $\delta_j(t_i)$ is in the subalgebra 
$\kk[t_1][t_2; \sigma_2, \delta_2]
\cdots [t_{j-1}; \sigma_{j-1},\delta_{j-1}]$. 
Let $G=\ZZ^n$ with order given in Example \ref{yyex4.1} and define 
$\deg_{G}(t_i)=e_i$ for each $1\leq i\leq n$. Then $A$ is a  
$\ZZ^n$-filtered algebra such that $\gr A\cong \SP$
with generators $\nu (\bt):=\{\nu (t_1),\cdots, \nu (t_n)\}$. By 
definition, for each $g\in G$, $F_{<g}$ is a the subspace of $A$ 
generated by all elements of the form $t_1^{d_1}\cdots t_n^{d_n}$ 
with $(d_1,\cdots,d_n)<g$.

\begin{lemma}
\label{yylem4.5}
Let $A$ be an AS regular algebra with a minimal generating set 
$\bx$ and let $\kk[t_1][t_2; \sigma_2, \delta_2]\cdots 
[t_n; \sigma_n,\delta_n]$ be a filtered skew polynomial ring 
realization of $A$. Let $w\in A$ be a normal element and let 
$\phi = \eta_w$. Then $\nu(\phi(t_j))\in \kk \nu(t_j)$ for all 
$1\leq j\leq n$. As a consequence, 
$\nu(\phi(x_i))\in \kk \nu(x_i)$ for all $1\leq i\leq g$.
\end{lemma}

\begin{proof} 
By definition, we have $\phi(t_j)=w^{-1} t_j w$ or
\[w \phi(t_j)=t_j w.\]
Applying $\nu$ and (L3), we obtain that 
\[\nu(w) \nu(\phi(t_j))=\nu(t_j) \nu(w).\]
Since $\nu(w)$ is a scalar multiple of a product of the $t_i$, then
\[ \nu(\phi(t_j))=\nu(w)^{-1}\nu(t_j) \nu(w)\] 
is a scalar multiple of $\nu(t_j)$. The assertion follows. The 
consequence follows since $\bx$ is a subset of $\bt$.
\end{proof}

Sometimes when $A$ has two different filtered skew polynomial 
ring realizations, $\phi$ must be a diagonal automorphism.

\begin{proposition}\label{yypro4.6}
Let $A$ be an AS regular algebra with a fixed minimal generating 
set $\bx$, which is a basis of $A_1$. Let $G=\ZZ^n$ with the 
order given in Example \ref{yyex4.1}. Suppose that $A$ can be 
realized as a $G$-filtered skew polynomial ring in two ways, one 
using generators $\bt$ and the other using generators $\bt'$, 
such that the following conditions are satisfied:
\begin{enumerate}
\item[(1)]
In the first realization with $G$-degree denoted by $\deg_1$, 
$\deg_{1}(t_j)=e_j$ and $\deg_{1}(x_i)<\deg_{1}(x_{i'})$
 when $i<i'$.
\item[(2)]
In the second realization with $G$-degree denoted by $\deg_2$, 
$\deg_{2}(t'_j)=e_j$ and 
$\deg_{2}(x_i) > \deg_{2}(x_{i'})$ when $i<i'$.   
\end{enumerate}
Let $w \in A$ be a normal element and let $\phi=\eta_w$. Then 
$\phi(x_i)\in \kk x_i$.

As a consequence, if $A$ is also PI, then $\Oz(A)$ is abelian.
\end{proposition}

\begin{proof}
By definition of a filtered skew polynomial ring $\bx\subset \bt$ 
and $\bx\subset \bt'$. For any fixed $i$, write 
$\phi(x_i)=\sum_{k} c_k x_k$. We claim that $c_k=0$ if $k\neq i$. 
Suppose to the contrary that $c_k\neq 0$ for some $k\neq i$. 
There are two cases: either $k>i$ or $k<i$. The proofs are similar, 
so we only consider the case when $k>i$ (and assume $k$ is largest 
among such integers $k$ with $c_k\neq 0$). In this case we use the 
first realization given in (1). Since $e_k>e_i$, by 
Lemma~\ref{yylem4.5}, $\nu(\phi(x_i))=e_k\not \in \kk e_i = 
\kk \nu(x_i)$, yielding a contradiction. The assertion follows.

The consequence follows from Lemma \ref{yylem1.9}(1).
\end{proof}

We can apply to the above results to down-up algebras. The group 
of graded algebra automorphisms of $A: = A(\alpha, \beta)$ was 
computed by Kirkman and Kuzmanovich \cite[Proposition 1.1]{KK}. 
If $\beta\neq \pm 1$, then every graded algebra automorphism 
of $A$ is diagonal. By \cite[(1.5.6)]{LMZ}, the Nakayama 
automorphism of $A$ is given by
\[\mu_A: x\mapsto -\beta x, \quad y\mapsto -\beta^{-1} y.\]
Let $\omega_1$ and $\omega_2$ be the roots of the characteristic equation
\[w^2-\alpha w-\beta=0\]
and let $\Omega_i=xy-\omega_i yx$ for $i=1,2$. It is easy to 
see that for $\{i,j\}=\{1,2\}$ we have
\[
x\Omega_i=\omega_j \Omega_i x, \quad
\omega_j y \Omega_i = \Omega_i y.
\]
Note that $A$ is PI if and only if both $\omega_1$ and $\omega_2$ 
are roots of unity. Algebra generators for the center of $A$ were 
computed by Zhao \cite[Theorem 1.3]{zhaoDU}.

\begin{corollary}\label{yycor4.7}
Let $A$ be a noetherian PI down-up algebra and $\phi\in \Oz(A)$. 
Then $\phi(x)\in \kk x$ and $\phi(y)\in \kk y$. As a 
consequence, $\Oz(A)$ is abelian.
\end{corollary}

\begin{proof} 
Let $A=A(\alpha,\beta)$ be PI. Set $z=\Omega_1$ so we have the 
following three relations
\[ xz = \omega_2 zx, \quad \omega_2y z = zy, \quad xy=\omega_1 yx+z.\]
We claim that $A$ has two different filtered skew polynomial ring 
realizations.

1. $A=\kk[x][z; \sigma_2][y;\sigma_3, \delta_3]$ where 
$\sigma_2: x\mapsto \omega_2\inv x$,
$\sigma_3: x\mapsto \omega_1\inv x, z\mapsto \omega_2\inv z$ and 
$\delta_3: x\mapsto -\omega_1^{-1}z, z\mapsto 0$. 
In this case $G=\ZZ^3$ with $\deg_1(x)=e_1$, $\deg_1(z)=e_2$ and 
$\deg_1(y)=e_3$. Thus $\deg_1(x)<\deg_1(y)$.

2. $A=\kk[y][z; \tau_2][x;\tau_3, \partial_3]$ where 
$\tau_2: y\mapsto \omega_2 y$,
$\tau_3: y\mapsto \omega_1 y, 
z\mapsto\omega_2 z$ and $\partial_3: y\mapsto z, z\mapsto 0$. 
In this case $G=\ZZ^3$ with $\deg_{2}(y)=e_1$, $\deg_{2}(z)=e_2$ 
and $\deg_{2}(x)=e_3$. Thus $\deg_{2}(y)<\deg_{2}(x)$.

By Proposition \ref{yypro4.6}, if $\phi=\eta_w$, then 
$\phi(x)\in\kk x$ and $\phi(y)\in \kk y$. The assertion follows 
from Lemma \ref{yylem1.9}(1). The consequence is clear.
\end{proof}

\begin{example}\label{yyexa4.8}
(1) Let $A=A(0,1)$. The center of $A$ is a polynomial ring 
generated by $x^2, y^2$ and $xy+yx$ so in this case $Z=Z(A)$ is 
regular. It follows from \eqref{E1.2.1} that $\rk_Z(A)=4$.

Let $\phi\in \Oz(A)$. Since $\phi(x^2)=x^2$ and $\phi(y^2)=y^2$,
then by Corollary \ref{yycor4.7}, $\phi(x)=\epsilon_1 x$ and 
$\phi(y)=\epsilon_2 y$ for $\epsilon_1, \epsilon_2\in \{1,-1\}$. 
Since $\phi$ preserves the central element $xy+yx$, 
$\epsilon_1=\epsilon_2$. Hence $\Oz(A)$ has order $2$. In 
particular, $\Oz(A)=\{\Id, \mu_A\} \cong \cyc_2$ where $\mu_A$ 
is the Nakayama automorphism of $A$. It is also clear that 
$\mu_A=\eta_{xy-yx}$.

By Lemma \ref{yylem3.2}, $\overline{Z}$ is generated by 
\[ 
t_1:=x^2\# \Id, \quad 
t_2:=y^2\# \Id, \quad t_3:=(xy+yx)\# \Id, \quad 
t_4=(xy-yx)\# \mu_A
\] 
and subject to the relation
\[t_4^2=t_3^2-4t_1t_2.\]
So $\overline{Z}$ has  an isolated singularity at the origin.

(2) Let $A=A(0,-1)$. In this case $A$ is CY. The center $Z$ of 
$A$ is generated by $x^4$, $y^4$, $\Omega_1\Omega_2$ and 
$\Omega_1^4$. A computation shows that
\[(\Omega_1^4)^2-2(\Omega_1^4)(\Omega_1\Omega_2)^2
+(\Omega_1\Omega_2)^4+16(\Omega_1)^4x^4y^4=0.\]
As a consequence, the  Hilbert series of $Z$ is 
$\frac{1-t^{16}}{(1-t^4)^3(1-t^8)}$. It follows from \eqref{E1.2.1} 
that the rank of $A$ over $Z$ is 16. 

By Corollary \ref{yycor4.7}, every element of $\Oz(A)$ is of the form
$\phi_{c,d}: x \mapsto cx, y \mapsto dx$ for some $c,d \in \kk$. Then 
\[ \Oz(A) = \{ \phi_{c,d} \mid c^4 = d^4 = c^2d^2 = 1 \} 
\iso \ZZ_4 \times \ZZ_2.\]
Note that $\eta_{x^2}=\phi_{1,-1}$, $\eta_{y^2}=\phi_{-1,1}$, and 
$\eta_{\Omega_1}=\phi_{-i,i}$. So $\Oz(A)$ is generated by 
$\eta_{x^2}$ and $\eta_{\Omega_1}$. 

By Lemma~\ref{yylem3.2}, we know that $\overline{Z}$ is spanned by 
$\{ a \# \eta_a \mid a \in A^{\Oz(A)} \text{ is normal in $A$}\}$. 
By the computation of $\Oz(A)$, we see that $A^{\Oz(A)}$ is 
generated as an algebra by $\{x^4, y^4, x^2y^2, xyxy, yxyx\}$. 
Hence, it is generated as an algebra by 
\[ \{x^4, y^4, x^2y^2,\Omega_1^2,\Omega_1\Omega_2,\Omega_2^2\}.\] 
We saw above that $x^4, y^4, \Omega_1 \Omega_2$ are central, and 
we can compute that 
$\eta_{x^2y^2}=\eta_{\Omega_1^2}=\eta_{\Omega_2^2}=\phi_{-1,-1}$. 
Therefore, $\overline{Z}$ is generated as an algebra by
\[
\{x^4 \# \Id, \, y^4 \# \Id, \, \Omega_1\Omega_2 \# \Id, 
\, x^2y^2 \# \phi_{-1,-1}, \, \Omega_1^2 \# \phi_{-1,-1}, 
\, \Omega_2^2 \# \phi_{-1,-1}\}.\]
\end{example}

For the rest of this section, we prove our characterization of 
skew polynomial rings in terms of their ozone groups 
(Theorem~\ref{yythm0.8}). We need a few lemmas.

Let $A$ be a PI domain. As in the proof of Lemma~\ref{yylem1.10}, 
let $N^\ast$ be the semigroup of nonzero normal elements in $A$. 
For each $w\in N^\ast$, let $\{w\}$ be the set of normal elements 
in $A$ that are equivalent to $w$ in the sense of the equivalence 
relation $\sim$ of \eqref{E1.10.1}. For each $\phi\in \Oz(A)$, let 
$\{\phi\}$ denote the set of normal elements $f \in A$ such that
$\phi=\eta_f$ (there exists such an element by 
Lemma~\ref{yylem1.9}). Let $[w]:=\{w\}\cup\{0\}$ and $[\phi]:=
\{\phi\}\cup \{0\}$.

\begin{lemma}
\label{yylem4.9}
Let $A$ be a ${\mathbb Z}$-graded PI domain that is a finite 
module over its center $Z$. Let $w \in N^\ast$ and 
$\phi = \eta_f$ for a normal element $f \in A$.
\begin{enumerate}
\item[(1)]
Suppose $w'\in N^{\ast}$. Then $[w]=[w']$ if and only if $w$ 
and $w'$ are equivalent.
\item[(2)]
$[w]$ and $[\phi]$ are graded $Z$-modules. 
\item[(3)]
$[w]=[\eta_{w}]$.
\item[(4)]
$\rk_Z([w])=1$.
\item[(5)]
Let $M$ be a $Z$-submodule of $A$ containing $w$. If 
$\rk_Z(M)=1$, then $M=[w]$. 
\end{enumerate}
\end{lemma}

\begin{proof}
(1) This follows since $[w] = \{w\} \cup \{0\} = \{w'\} 
\cup \{0\} = [w']$.

(2) Let $x,y \in [w]$. If $x+y = 0$ then $x+y \in [w]$. 
Assume $x+y \neq 0$. Then $wz_1=xz_2$ and $wz_3=yz_4$ for 
some $z_i \in Z^\ast$. Hence,
\[ (x+y)z_2z_4 = (xz_2)z_4 + (yz_4)z_2
= (wz_1)z_4 + (wz_3)z_2 = w(z_1z_4 + z_3z_2).\]
Since $A$ (and hence $Z$) is a domain, then 
$z_1z_4 + z_3z_2 \neq 0$. Thus, $x+y \sim w$ so $x+y \in [w]$. 
Now if $z \in Z^\ast$, then 
\[(xz)z_2 = (xz_2)z = (wz_1)z = w(z_1z)\]
so $xz \sim w$. Thus, $xz \in [w]$. It follows that $[w]$ is 
a (graded) $Z$-module. The proof for $[\phi]$ is similar.

(3) Let $x \in \{w\}$, so $wz_1=xz_2$ for some $z_i \in Z^\ast$. 
Then for any $a \in A$ we have 
\[\eta_w(a) = waw\inv = (wz_1)a(wz_1)\inv 
= (xz_2)a(xz_2)\inv = xax\inv = \eta_x(a).\]
Thus, $\eta_x = \eta_w$ so $x \in \{\eta_w\}$. 

Conversely, suppose $y \in \{\eta_w\}$. Then $\eta_w=\eta_y$. 
Let $a \in A$. Then
\[ waw\inv = \eta_w(a) = \eta_y(a) = yay\inv.\]
Consequently, $(y\inv w)a(y\inv w)\inv = a$, so 
$y\inv w = z \in Z^*$. Thus, $w = yz$, so $y \in \{w\}$.

(4) Let $Q = Z(Z \setminus \{0\})^{-1}$ be the field of 
fractions of $Z$ and let $[w]_{Q}$ be $[w]\otimes_Z Q$. 
Then $\rk_Z([w])=\rk_Q([w]_Q)$. It is easy to see that 
$[w]_Q$ is cyclic with generator $w$.

(5) Clearly $[w] \subset M$. Let $x \in M$ be a generator 
of $M\otimes_Z Q$ as a $Q$-module. Then $xq = w$ for some 
$q \in Q$. After clearing fractions we have $xz_1 = wz_2$ 
for some $z_1,z_2 \in Z$. Consequently, $x \in [w]$.
\end{proof}

\begin{lemma}\label{yylem4.10}
Suppose $\kk$ is algebraically closed.
Let $A$ be a ${\mathbb Z}$-graded domain that is module-finite 
over its center. Suppose $\Oz(A)$ is a finite abelian group. 
Then the sum $\sum_{\phi\in \Oz(A)} [\phi]$ is a direct sum 
of $Z$-modules.
\end{lemma}
\begin{proof} 
Let $G$ be the character group of $\Oz(A)$. Then there is a 
natural inner-faithful $G$-coaction on $A$ so that $A$ is 
$G$-graded with decomposition
\[ A=\bigoplus_{\chi\in G} A_{\chi}\]
where $A_{\chi}=\{ a\in A \mid g(a) = \chi(g) a \text{ for all } 
g\in \Oz(A)\}$. Since the $G$-coaction is inner-faithful, 
$A_{\chi}\neq 0$ for all $\chi\in G$. 

Suppose to the contrary that $\sum_{\phi\in \Oz(A)} [\phi]$ is 
not a direct sum. Then we can choose nonzero elements $f_i$, 
each belonging to some $[\phi]$ such that
\begin{align}\label{E4.10.1}
\sum_{i=1}^n f_i=0.
\end{align}
Since all $f_i$ are nonzero, we see that $n\geq 2$. We choose 
$f_i$ so that $n$ is minimal. Multiplying by a nonzero element 
in $[\eta_{f_n}^{-1}]$, we can assume that $f_n$ is central. We 
write \eqref{E4.10.1} as
\[\sum_{i=1}^{n-1} f_1=-f_n.\]
Choose $\chi\in G$ so that $\chi(\eta_{f_1})\neq 1$. Since 
$A_{\chi}\neq 0$, let $a$ be a nonzero element in $A_{\chi}$. 
Then 
\begin{align*}
0& =(-f_n)a-a(-f_n)=
\sum_{i=1}^{n-1} (f_i a-a f_i)\\
&=\sum_{i=1}^{n-1} (f_i a -f_i \eta_{f_i}(a))
=\sum_{i=1}^{n-1} (f_i a- f_i \chi(\eta_{f_i}) a).
\end{align*}
Since $A$ is a domain, after cancelling out $a$, we obtain
that $\sum_{i=1}^{n-1} f_i (1-\chi(\eta_{f_i}))=0$ which 
contradicts to the minimality of $n$.
\end{proof}

The following lemma should be well-known.

\begin{lemma}
\label{yylem4.11}
Let $A$ be a noetherian connected graded algebra that is 
generated in degree one. Assume $A$ has finite global 
dimension. If $A$ is generated as a $\kk$-algebra by normal 
elements, then $A$ is isomorphic to $\SP$.
\end{lemma}

\begin{proof}
In this case $A$ has enough normal elements in the sense of 
\cite[p.392]{Zh2}. By \cite[Theorem 1]{Zh2}, $A$ is a domain.

Since every normal element is a sum of homogeneous normal 
elements, $A$ is generated as an algebra by homogeneous normal 
elements. Since $A$ is generated in degree 1, every element in 
$A_1$ is a sum of degree 1 normal elements and so
$A$ is generated by degree 1 normal elements. Pick any 
$\kk$-linear basis of $A_1$, say $\{x_1,\cdots,x_n\}$, 
consisting of normal elements. We claim that, for all $i<j$, 
$x_j x_i=p_{ij} x_i x_j$ for some $p_{ij}\in \kk^{\times}$. By 
\cite[Lemma 5.2(3)]{KWZ2}, $A/(x_i)$ is noetherian of finite 
global dimension. By \cite[Theorem 1]{Zh2}, $A/(x_i)$ is a 
domain. By definition, in $A/(x_i)$, we have
\[ 0=x_i x_j=x_j (\eta_{x_j}(x_i))\]
and $x_j\neq 0$ in $A/(x_i)$. Then $\eta_{x_j}(x_i)=0$ 
in $A/(x_i)$ which implies that $\eta_{x_j}(x_i)=p^{-1}_{ij} x_i$ 
for some nonzero scalar $p_{ij}$. Then $x_i x_j=
x_j (\eta_{x_j}(x))$ implies that $x_j x_i=p_{ij} x_i x_j$. 
This proves the claim.

Since $A$ is generated by $\{x_1,\cdots,x_n\}$ and $x_j x_i=
p_{ij} x_i x_j$ in $A$ for all $i<j$, there is a surjective 
algebra morphism $\Phi: \SP\to A$ such that 
$\dim (S_{\mathbf p})_1=\dim A_1=n$. 
Since $A$ is noetherian and has finite global dimension,
its Hilbert series $h_A(t)$ is of the form $\frac{1}{p(t)}$ for 
some polynomial $p(t)$. Since $\SP$ has Hilbert series 
$h_{\SP}(t)=\frac{1}{(1-t)^n}$ and $A$ is a cyclic (hence 
finitely generated) graded module over $\SP$, we also have 
$h_{A}(t)=\frac{q(t)}{(1-t)^n}$ for some polynomial $q(t)$. 
Then the equation
\[\frac{1}{p(t)}=\frac{q(t)}{(1-t)^n}\]
forces that $p(t)\mid (1-t)^n$. Since $A$ is connected graded, 
we see $p(0)=1$ and so $p(t)=(1-t)^m$ for some $m\leq n$. 
Expanding the rational function $\frac{1}{p(t)}=h_{A}(t)$, we 
see that $\dim A_1=m$. Thus $m=n$ so $\Phi$ is an isomorphism.
\end{proof}

We note that in the following lemma we do not make the 
assumption that $A$ is generated in degree one. This will be 
crucial in the sequel when we apply this result to the 
Zariski cancellation problem (Theorem~\ref{yythm4.14}).

\begin{lemma}
\label{yylem4.12}
Suppose $\kk$ is algebraically closed.
Let $A$ be a noetherian PI AS regular algebra. If $\Oz(A)$ is 
abelian and $\left|\Oz(A)\right|=\rk_Z(A)$, then $A$ is 
generated by normal elements.
\end{lemma}

\begin{proof}
Similar to the proof of Lemma~\ref{yylem4.10}, let $G$ be the 
character group of $\Oz(A)$. So, we have a decomposition of 
$Z$-modules,
\[ A=\bigoplus_{\chi\in G} A_{\chi}\]
with each $A_{\chi}$ nonzero. 
Since $\rk_Z(A)=\left|\Oz(A)\right|=|G|$, we have 
$\rk_Z(A_{\chi})=1$ for each $\chi\in G$. By Lemma 
\ref{yylem4.10}, we also have a decomposition of $Z$-modules
\[ M=\bigoplus_{\phi\in \Oz(A)} [\phi]\subseteq A\]
where each $[\phi]\neq 0$. We claim that for each 
$\phi\in \Oz(A)$, there is a $\chi\in G$ such that 
$[\phi]=A_{\chi}$. Since $\Oz(A)$ is abelian, for any normal 
elements $w_1,w_2 \in A$, 
$\eta_{w_1w_2} =\eta_{w_1}\eta_{w_2}=\eta_{w_2}\eta_{w_1}
=\eta_{w_2w_1}$. 
By Lemma~\ref{yylem4.9}, $[w_1w_2]=[w_2w_1]$ which implies
that $\eta_{w_1}$ maps $[w_2]$ to $[w_2]$. Therefore 
$\Oz(A)$ acts on $[w_2]$ and so there is a decomposition 
of $Z$-modules
\[ [w_2]=\bigoplus_{\chi\in G} [w_2]_{\chi}.\]
Since $\rk_Z(A)=\left|\Oz(A)\right|$, therefore 
$\rk_{Z}([w_2])=1$. In particular, $[w_2]=[w_2]_{\chi_0}$
for one particular $\chi_0\in G$. Hence $[w_2]\subseteq 
A_{\chi_0}$. By Lemma~\ref{yylem4.9}(5), $[w_2]=A_{\chi_0}$. 
Therefore, we see that 
\[ A=\bigoplus_{\chi\in G} A_{\chi}=
\bigoplus_{\phi\in \Oz(A)} [\phi]\subseteq A.\]
Hence, $A$ is generated by normal elements and 
every element in $A_{\chi}$ is normal. 
\end{proof}

We can now prove Theorem~\ref{yythmintro0.8} from the introduction.

\begin{theorem}[Theorem~\ref{yythmintro0.8}]
\label{yythm0.8}
Suppose $\kk$ is algebraically closed. Let $A$ be a noetherian 
PI AS regular algebra that is generated in degree one. The 
following are equivalent.
\begin{enumerate}
\item[(1)] 
$A$ is isomorphic to $S_{\mathbf p}$.
\item[(2)] 
$\Oz(A)$ is abelian and $\left|\Oz(A)\right| = \rk_Z(A)$.
\item[(3)] 
$A$ is generated by normal elements.
\item[(4)] $\Oz(A)$ is abelian and $A^{\Oz(A)}=Z(A)$.
\end{enumerate}
\end{theorem}
\begin{proof}
(3) $\Rightarrow$ (1) This is Lemma~\ref{yylem4.11}.

(1) $\Rightarrow$ (2) This follows from Lemma~\ref{yylem4.2} 
and Proposition~\ref{yypro3.9}(2).

(2) $\Rightarrow$ (3) This follows from Lemma~\ref{yylem4.12}.

(2) $\Leftrightarrow$ (4) This follows from Proposition~\ref{yypro3.9}.
\end{proof}

As another application of Theorem~\ref{yythm0.8}, and Lemma 
\ref{yylem4.12} in particular, we consider the Zariski 
cancellation for skew polynomial rings.

\begin{lemma}\label{yylem4.13}
Let $A=\kk_\bp[x,y,z]$ be a PI skew polynomial ring. Let $B$ 
be a connected graded algebra, not necessarily generated in 
degree one. Suppose $A[t] \iso B[t]$. Then $B$ is a skew 
polynomial ring.
\end{lemma}
\begin{proof}
It is clear that $B$ is a noetherian PI domain and that 
$\PIdeg(A)=\PIdeg(B)$. Furthermore, by Lemma~\ref{yylem1.11},
\[ \Oz(A) = \Oz(A[t]) = \Oz(B[t]) = \Oz(B).\]
Thus, by Theorem~\ref{yythm0.8}, $\Oz(B)$ is a group of order 
$\rk_{Z(A)}(A) = \rk_{Z(B)}(B)$. Consequently, by Lemma~\ref{yylem4.12}, $B$ is generated 
by normal elements. As $B$ is connected 
graded, we may assume that the normal generators are homogeneous. 
We claim that these elements can be chosen so that they skew 
commute.

Since $A$ is Artin--Schelter regular, then so are $A[t]$ and 
$B[t]$. Since $t$ is a degree one central, regular element, 
$B$ is necessarily Artin--Schelter regular as well. Thus, by 
\cite[Proposition 1.1]{steph}, $B$ is generated either by two or 
three homogeneous elements.

Suppose $B$ is generated by two (homogeneous) elements $x$ and 
$y$. Then there are two cases, either $\deg(x)=\deg(y)$ or else 
$\deg(x)<\deg(y)$. In the first case, we may take the common 
degrees to be $1$ and the result follows from Lemma 
\ref{yylem4.11}. In the second case, since $\deg(yxy\inv)=\deg(x)$, 
it follows that $yxy\inv \in \kk x$, so $yx=p xy$ for some 
$p \in \kk^\times$. Note that both of these cases prove that 
$B \iso \kk_p[x,y]$ for some $p \in \kk^\times$, a contradiction 
since $\GKdim B = 3$.

Therefore, $B$ is generated by three elements, $x$, 
$y$, and $z$. As above, we may assume that these elements are 
homogeneous. Moreover, again by \cite[Proposition 1.1]{steph},
the Hilbert series of $B$ is equal to that of a weighted
polynomial ring in three variables where the generators have degree $\deg(x)$, $\deg(y)$, and $\deg(z)$.

(Case 1): Suppose that $\deg(x) < \deg(y) 
< \deg(z)$. The same argument as above shows that $yx=p xy$ and $zx=qxz$ for some $p,q \in \kk^\times$.
If $\deg(x) \nmid \deg(y)$, then it follows immediately that $zy=ryz$ for some $r \in \kk^\times$.
Otherwise, set $k=\deg(y)/\deg(x)>1$. Then we have $zy=(ay+bx^k)z$. 
If $b=0$ then we are done so suppose $b \neq 0$. Now
\[ pqa(xyz) + pqb (x^{k+1} z) = z(yx) = (zy)x = pqa(xyz) + qb(x^{k+1}z).\]
Hence, $p=1$. If $a \neq q^k$, then set $y'=y+(b/(a-q^k))x^k$ so then $zy'=ay'z$. 
In the case that $a=q^k$, then we observe that $\eta_z(x)=qx$ and $\eta_z(y) = q^k y + bx^k$.
It follows that $\eta_z^n(y) = q^{nk}y + nq^k b x^k$
and so $\eta_z$ has infinite order, but
this contradicts Theorem~\ref{yythm0.7}.

(Case 2): Suppose the elements are chosen so that $\deg(x) = \deg(y) < \deg(z)$. 
The argument from Lemma~\ref{yylem4.11} shows that $yx=pxy$ for some $p \in \kk^\times$.
Now using similar arguments as above gives that
$zx=(ax+by)z$ and $zy=(cx+dy)z$. If $p=1$, then since $B$
is PI we can diagonalize $\eta_z$ and obtain the result.
Now assume $p \neq 1$. Then we have
\begin{align*}
z(yx)    &= p( ac x^2 + (ad+pbc)xy + bdy^2)z \\
(zy)x    &= (ac x^2 + (bc+pad)xy + bdy^2)z.
\end{align*}
We may regrade $B$ so that $\deg(x)=\deg(y)=1$ and $\deg(z)=k>1$. If $k>2$, then a Hilbert series argument shows that $\dim B_{k+2}=k+6$ and a basis of $B_{k+2}$ is $\{x^iy^j z^\epsilon : i+j+\epsilon=k+2, \epsilon \in \{0,1\}\}$. In case $k=2$, $z^2$ is an additional basis element. In either case, $x^2z$, $xyz$, and $y^2z$ are linearly independent.
Since $p \neq 1$, $ac=bd=0$. If $b=c=0$ then we are done, so assume
that $a=d=0$. Then from the above we have $p^2bc=bc$ so $p^2=1$.
Thus, $p=-1$, $zx=byz$ and $zy=cxz$ where $b,c$ are roots of unity.
But then $x$ and $y$ are not normal, a contradiction.

(Case 3):
Suppose the elements are chosen so that $\deg(x) < \deg(y) = \deg(z)$. 
We have $yx=pxy$ and $zx=qxz$ for some $p,q \in \kk^\times$. 
Using the reasoning as above, we have 
$zy=(ay + bz+cx^k)z$ and $yz = (a'z+b'y+c'x^k)y$
where $k=\deg(y)/\deg(x)$ and $c=c'=0$ if $k$ is not an integer. 
If $k$ is an integer, we may assume $\deg(x) = 1$. Further, we have
\[ yz = a'zy+b'y^2+c'x^ky = a'(ayz + bz^2+cx^kz) + b'y^2 + c'x^k y.\]
A Hilbert series argument shows that $\dim B_{2k} = 6$ and so a basis of $B_{2k}$ is $\{  x^{2k}, y^2, z^2, x^ky, x^kz, yz\}$.
It follows that $a'=a\inv$ and $b=c=b'=c'=0$. Thus, $zy=ayz$.

Otherwise, $k$ is not an integer and we can assume $c=c'=0$.
But then a similar Hilbert series argument shows that $\dim B_{2k}=3$ so $yz$, $y^2$, and $z^2$ are all linearly independent. The result follows by the above computation.

(Case 4): Suppose the elements are chosen so that $\deg(x)=\deg(y)=\deg(z)$. 
Then we may as well take these degrees to be 1 so the result 
follows from Lemma~\ref{yylem4.11}.
\end{proof}

\begin{theorem}\label{yythm4.14}
Let $A=\kk_\bp[x_1,x_2,x_3]$ be a PI skew polynomial ring. 
Then $A$ is cancellative in the class of connected graded 
algebras.
\end{theorem}
\begin{proof}
Suppose $A[t] \cong B[t]$ for some connected graded algebra $B$.
By Lemma~\ref{yylem4.13}, $B \iso \kk_\bq[u_1,u_2,u_3]$ for 
some parameters $\bq=(q_{ij})$. Now there is no loss in taking 
$\deg(u_i)=1$. Now both $A[t]$ and $B[t]$ are skew polynomial 
rings. In particular, $A[t]=\kk_{\bp'}[x_1,x_2,x_3,x_4]$ where 
$p_{ij}'=p_{ij}$ for $i,j \leq 3$ and $p_{4j}=p_{i4}=1$ for all 
$i,j$. Similarly, $B[t]=\kk_{\bq'}[u_1,u_2,u_3,u_4]$. By 
\cite[Theorem 2.4]{Ga}, $\bq'$ is a permutation of $\bp'$ 
(in the sense that there exists $\sigma \in \mathcal{S}_4$ 
such that $q_{ij}' = p_{\sigma(i)\sigma(j)}'$ for all $i,j$).
Consequently, after possibly switching $x_4$ (or $u_4$) with 
another central generator, we have that $\bq$ is a permutation 
of $\bp$. Thus, $A \iso B$.
\end{proof}

\begin{corollary}[Theorem~\ref{yythmintro4.14}]
\label{yycor4.15}
Let $A=\kk_p[x,y,z]$ with $p$ a root of 1. Then 
$A$ is cancellative in the class of connected graded algebras.
\end{corollary}

\begin{question}\label{yyque5.6}
Is $A=\kk_\bp[x_1,x_2,x_3]$ cancellative 
in the class of the class of all $\kk$-algebras?
\end{question}

\section{Comments, examples, and questions}
\label{yysec5}

In this section we present some questions which we hope will 
inspire further investigation into the ozone group and its 
connection to PI Artin--Schelter regular algebras.

The first example shows that the connected hypothesis in 
Question~\ref{yyque0.10} is necessary. It also gives an example of an ozone group of infinite order.

\begin{example}
\label{yyexa5.1}
Let $A$ be the preprojective algebra of the extended Dynkin 
quiver $\widetilde{A}_2$. It is known that $A$ is CY of 
global and GK dimension 2, PI, and noetherian. 

Denote the trivial idempotents of $A$ by $e_0,e_1,e_2$. We 
denote the paths by $a_0,a_1,a_2$ and the dual paths by 
$a_0^*,a_1^*,a_2^*$. That is, $A$ is the quotient of the 
path algebra of the quiver
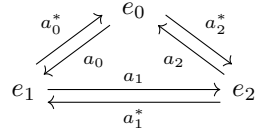
\begin{center}
\begin{tikzcd}
& e_0 \arrow[ld, "a_0"] \arrow[rd, "a_2^*", shift left=2] & \\
e_1 \arrow[rr, "a_1"] \arrow[ru, "a_0^*", shift left=2] &  
& e_2 \arrow[lu, "a_2"] \arrow[ll, "a_1^*", shift left=2]
\end{tikzcd}
\end{center}
by the relation $\sum \left(a_ia_i^*-a_i^*a_i\right)$.

It is known that the center of $A$ is generated as a $\kk$-algebra by
\begin{align*}
p &= a_0a_1a_2 + a_1a_2a_0 + a_2a_0a_1 \\
q &= a_0^*a_2^*a_1^* + a_2^*a_1^*a_0^* + a_1^*a_0^*a_2^* \\
z &= a_0a_0^* + a_1a_1^* + a_2a_2^*.
\end{align*}
These satisfy the relation $pq=z^3$. In particular, the center 
is a Kleinian singularity.

Choose any scalars $\beta_0, \beta_1, \beta_2 \in \kk^\times$ such that $\beta_0 \beta_1 \beta_2 = 1$. Then there is an ozone automorphism $\theta_{\beta_0, \beta_1, \beta_2}$ defined by $\theta_{\beta_0, \beta_1, \beta_2}(e_i) = e_i$, $\theta_{\beta_0, \beta_1, \beta_2}(a_i) = \beta_i a_i$, and $\theta_{\beta_0, \beta_1, \beta_2}(a_i^*) = \beta_i\inv a_i^*$. Note that $\theta_{\beta_0, \beta_1, \beta_2}$ is simply conjugation by the invertible element $\beta_0 e_0 + e_1 +  \beta_1\inv e_2$. 
In particular, since $\beta_0, \beta_1 \in \kk^\times$ can be chosen arbitrarily, this shows that $\Oz(A)$ has infinite order.

Let $\phi$ be the automorphism of $A$ determined by 
$\phi(e_i)=e_{i+1}$, so in particular $\phi(a_i)=a_{i+1}$ 
and $\phi(a_i^*)=a_{i+1}^*$. Clearly $\phi \in \Oz(A)$. 
Let $\omega$ be a primitive third root of unity and set $\psi = \theta_{\omega, \omega^2, 1}$. Then $\psi\phi(a_0)=\omega^2 a_0 \neq \omega a_0 
= \phi\psi(a_0)$, so $\Oz(A)$ is not abelian.
\end{example}

The Nakayama automorphism of a PI AS regular algebra $A$ 
necessarily belongs to $\Oz(A)$. Hence, if $\Oz(A)$ is trivial 
then $A$ is CY. Note in Theorem~\ref{yythm0.6}, we classified 
all such quadratic $A$ of global dimension 3 with trivial ozone group.
Furthermore, by Lemma~\ref{yylem3.3}, $A \# \kk \Oz(A)$ is CY 
if and only if $\Oz(A)$ has trivial homological determinant. 
This leads to the following questions.

\begin{question}\label{yyque5.2}
Let $A$ be a PI AS regular algebra.
\begin{enumerate}
\item \label{yyque5.2.1}
Is there a connection between $Z(A)$ having an isolated 
singularity and $\Oz(A)$ being trivial?
\item \label{yyque5.2.2}
What role does the Nakayama automorphism play in the study 
of the ozone group and the center?
\item \label{yyque5.2.3}
Assume $\Oz(A)$ is abelian. If the $\Oz(A)$-action on $A$ 
has trivial homological determinant, then is $A$ CY? 
One can also ask the converse. If $A$ is CY, does the 
action of $\Oz(A)$ on $A$ necessarily have trivial 
homological determinant?
\item \label{yyque5.2.4}
Recall that $Z'$ denotes the center of $A^{\Oz(A)}$ and 
$\overline{Z}$ denotes the center of $A\#\kOz(A)$.
Under what extra hypotheses on $A$ are $Z'$ and $\overline{Z}$ 
isomorphic as graded algebras? 
(See Proposition~\ref{yypro3.5} for a partial result.)
\end{enumerate}
\end{question}

\subsection{The center as an iterated fixed ring}

In \cite{CGWZ2}, the present authors made use of the ozone group to study the center $\ZSP$ of the skew polynomial ring $\SP$. In this case,
$\ZSP = \SP^{\Oz(\SP)}$, so one may use invariant theory to study $\ZSP$. 
Unfortunately, if $A \neq \SP$, then $A^{\Oz(A)} \neq Z(A)$, since $A$ will have rank $\left|\Oz(A)\right|$ over $A^{\Oz(A)}$.
In the worst case, $\Oz(A)$ is trivial, and so $A^{\Oz(A)} = A$.

However, in some cases, we can let $G = \Gal(A^{\Oz(A)}/ Z(A))$ and get $(A^{\Oz(A)})^G = Z(A)$, and so can use invariant theory to study $Z(A)$ via this iterative process. We present some examples below.
One can also imagine cases in which $Z(A)$ is obtained by taking a fixed ring after more than two iterations, although we do not have any explicit examples of this.

\begin{example}
Recall the quantum Heisenberg algebra $H_q$ from \eqref{E1.2.2}, where $q$ is a primitive $\ell$th root of unity such that $3 \nmid \ell$. Let $Z=Z(H_q)$. By Lemma~\ref{yylem1.3}, $Z$ is generated by $x^\ell$, $y^\ell$, $z^\ell$, and $\Omega z$ where $\Omega=xy-q^{-2}yx$.

In this case, by Proposition~\ref{yypro1.4}, we have $O=\Oz(H_q)= \grp{\phi} \iso \cyc_\ell$, where $\phi(x) = qx$, $\phi(y) = q\inv y$, and $\phi(z) = z$. An easy computation using Molien's theorem shows that $H_q^O$ is generated by $x^\ell$, $y^\ell$, $xy$, and $z$, subject to a single relation of degree $3\ell$. 
Using $\eqref{E1.2.1}$, we see that the rank of $A$ over its center $Z$ is
\[
\rk_Z(A)=(h_A(t)/h_Z(t))\restrict{t=1} = \left.\frac{1/(1-t)^3}{(1-t^{3\ell})/((1-t^{\ell})^3(1-t^3))}\right|_{t=1} = \ell^2.
\]

Set $G=\Gal(H_q^O/Z)$. There is an automorphism $\tau \in G$ defined by $\tau(x^\ell)=x^\ell$, $\tau(y^\ell)=y^\ell$, $\tau(z)=q z$, and $\tau(xy)=q xy$ (so also $\tau(yx)=q yx$). 
In particular, $G=\grp{\tau}$. Then it is easily verified that $(H_q^O)^G = Z$.
\end{example}

Recall also the down-up algebra $A := A(\alpha, \beta)$ from \eqref{E4.3.1}. While the generators of $Z(A)$ were computed by Zhao \cite{zhaoDU}, there is no general presentation given for $Z(A)$. In particular, it is not known whether the center of a PI down-up algebra is Gorenstein. The next two examples demonstrate that Galois groups may be useful in studying this problem.

\begin{example}
Let $A=A(0,1)$ be the down-up algebra from Example~\ref{yyexa4.8}(1). In this case, $\Oz(A)=\{1,\mu_A\}$ where $\mu_A: x \mapsto -x, y \mapsto -y$. Let $a=x^2$, $b=y^2$, $c=xy$, and $d=yx$. A straightforward computation shows that
\[ B = A^{\Oz(A)} \iso \frac{\kk[a,b,c,d]}{(ab-cd)}.\]
There is an automorphism $\sigma$ of $B$ given by $\sigma(a)=a$, $\sigma(b)=b$, $\sigma(c)=d$, and $\sigma(d)=c$. Moreover, $G=\Gal(B/Z(A))=\{1,\sigma\}$. Furthermore, $B^G = \kk[a,b,c+d]=Z(A)$.
\end{example}

\begin{example}
\label{yyexa5.10}
Let $A=A(0,-1)$ be the down-up algebra from Example 
\ref{yyexa4.8}(2) so that $\Oz(A) \iso \cyc_4 \times \cyc_2$. 
Here $B = A^{\Oz(A)}$ is generated by $a=x^4$, $b=y^4$, 
$c=x^2y^2$, $d=xyxy$, and $e=yxyx$. Hence, $B$ is commutative. 
It is not difficult to see that $G=\Gal(B/Z(A))=\{1,\tau\}$ 
where 
\[ \tau(a)=a, \tau(b)=b, \tau(c)=-c, \tau(d)=e, \tau(e)=d.\]
As in the previous example, $B^G=Z(A)$. 

But in this example we can make an additional observation. 
By \cite[Theorem 1.5]{KK}, 
$\hdet(\sigma)=\left(\det \sigma\restrict{A_1}\right)^2$ for 
any diagonal map $\sigma \in \Autgr(A)$. Consequently, $\Oz(A)$ 
acts by trivial homological determinant on $A$ and so $B$ is 
(AS) Gorenstein \cite[Theorem 3.3]{JoZ}. Since $B$ is 
commutative and $\det \tau\restrict{B_1} = 1$, then $Z(A)=B^G$ 
is Gorenstein.
\end{example}

\begin{question}
\label{yyque5.11}
Let $A$ be a PI down-up algebra with $O=\Oz(A)$ and $Z=Z(A)$. By 
\cite[Corollary 6.6]{CGWZ2}, $A^O$ is Gorenstein if and only if 
$\fb=\Omega_1\Omega_2 \in A^O$. Moreover, $\fb \in Z$ implies 
that $\fb \in A^O$. Can the iterative construction above be used 
to prove that $\fb \in A^O$ implies that $\fb \in Z$?
\end{question}

\subsection{Quantum thickenings of the ozone group}

As mentioned in the previous subsection, even when $A^{\Oz(A)} \neq Z(A)$, it may be possible to obtain the center as the fixed ring of $A^{\Oz(A)}$ by $\Gal(A^{\Oz(A)}/Z(A))$ (or perhaps after subsequent iterations). Another potential avenue to study the center $Z(A)$ using invariant theory would be to find a semisimple Hopf algebra $H$ such that $A^H = Z(A)$.

For a Hopf algebra $H$, let $G(H)$ denote the group of grouplike elements of $H$.

\begin{definition}
\label{yydef5.12}
Let $G$ be a group acting faithfully on an algebra $A$. Suppose that $H$ is a semisimple Hopf algebra acting inner-faithfully on $A$ with group of grouplikes $G(H)$. We say that the $H$-action on $A$ is a \emph{quantum thickening of the $G$-action on $A$}
if $G(H)$ if there is an isomorphism $G \to G(H)$ which preserves the actions on $A$. We say simply that $H$ is a \emph{quantum thickening of $G$} if the actions are clear.
\end{definition}

If we hope to obtain $Z = A^H$, it is reasonable to ask that $H$ be a quantum thickening of $\Oz(A)$. The following question is a more specific version of \cite[Question 6.2(4)]{CGWZ2}.

\begin{question}
\label{yyque5.13}
Given a PI AS regular algebra $A$, is there a quantum thickening $H$ of $\Oz(A)$ such that $A^H=Z(A)$?
\end{question}

Unfortunately, we do not have any examples of quantum thickenings which achieve an affirmative answer to Question~\ref{yyque5.13}.

\begin{example}
\label{yyexa5.14}
Let $A$ be the down-up algebra $A(0,1)$ as in \eqref{E4.3.1}. 
By Example~\ref{yyexa4.8}(1), $\Oz(A) \iso \cyc_2$ where the 
only nontrivial ozone automorphism scales both $x$ and $y$ by 
$-1$. Let $H$ be a quantum thickening of $\Oz(A)$ with 
$\dim(H)=4$. Then $H$ is either a group algebra or a dual of 
a group algebra and hence $A^H \neq Z(A)$.
\end{example}

\begin{example}
\label{yyexa5.15}
Let $A$ be the down-up algebra $A(0,-1)$ as in \eqref{E4.3.1}. 
By Example~\ref{yyexa4.8}(2), $\Oz(A) \iso \cyc_4 \times \cyc_2$. 
There is an action of the dimension 8 Kac--Palyutkin Hopf algebra 
$H$ on $A$ in which $G(H)=\Oz(A)$, but one can verify that 
$A^H \neq Z(A)$. Similarly, there are actions when $\dim(H)=16$ 
and $G(H)=\Oz(A)$, but again we find that $A^H \neq Z(A)$.
\end{example}

\providecommand{\bysame}{\leavevmode\hbox to3em{\hrulefill}\thinspace}
\providecommand{\MR}{\relax\ifhmode\unskip\space\fi MR }
\providecommand{\MRhref}[2]{%
  \href{http://www.ams.org/mathscinet-getitem?mr=#1}{#2}
}
\providecommand{\href}[2]{#2}

\end{document}